\documentclass{amsart}
\usepackage{amssymb}
\usepackage{graphicx}
\usepackage{mathtools}
\usepackage{manfnt}

\theoremstyle{plain}
\newtheorem{theorem}{Theorem}[section]
\newtheorem{conjecture}[theorem]{Conjecture}
\newtheorem{proposition}[theorem]{Proposition}
\newtheorem{lemma}[theorem]{Lemma}
\newtheorem{corollary}[theorem]{Corollary}

\newtheorem{problem}[theorem]{Problem}
\theoremstyle{definition}
\newtheorem{definition}[theorem]{Definition}
\theoremstyle{remark}
\newtheorem{remark}[theorem]{Remark}
\newtheorem{example}[theorem]{Example}

\def\Z{\mathbb{Z}}
\def\R{\mathbb{R}}
\def\C{\mathbb{C}}
\def\P{\mathbb{P}}
\def\F{\mathbb{F}}

\def\M{\mathcal{M}}
\def\O{\mathcal{O}}
\def\J{\mathcal{J}}

\def\conv{\mathrm{conv}}
\def\inte{\mathrm{int}}
\def\bnd{\mathrm{bnd}}
\def\refi{\mathrm{ref}}

\def\coh{\mathrm{coh}}
\def\lift{\mathrm{lift}}

\def\cube{\mbox{\,\mancube\,}}

\title{Positive Grassmannian and polyhedral subdivisions}

\author{Alexander Postnikov}
\address{Department of Mathematics, Massachusetts Institute of Technology,
77 Massachusetts Avenue, Cambridge MA 02139}
\email{apost@math.mit.edu}

\date{February 27, 2018}


\begin{document}

\begin{abstract}
The nonnegative Grassmannian is a cell complex with rich 
geometric, algebraic, and combinatorial structures.
Its study involves interesting combinatorial objects, such as 
positroids and plabic graphs.
Remarkably, the same combinatorial structures appeared in many 
other areas of mathematics and physics, e.g., in 
the study of cluster algebras, scattering amplitudes, and solitons.
We discuss new ways to think about these structures.
In particular, we identify plabic graphs and more general Grassmannian graphs 
with polyhedral subdivisions
induced by 2-dimensional projections of hypersimplices.
This implies a close relationship between the positive Grassmannian
and the theory of fiber polytopes and the generalized Baues problem.
This suggests natural extensions of objects related 
to the positive Grassmannian.
\end{abstract}

\maketitle

\section{Introduction}
\label{sec:in}

The geometry of the Grassmannian $Gr(k,n)$ is related to combinatorics
of the hypersimplex $\Delta_{kn}$.  Gelfand, Goresky, MacPherson, Serganova 
\cite{GGMS} studied the hypersimplex
as the moment polytope for the torus action on the complex Grassmannian.
In this paper we highlight new links between
geometry of the {\it positive Grassmannian\/} 
and combinatorics of the hypersimplex $\Delta_{kn}$.

\smallskip

\cite{GGMS} studied the {\it matroid stratification\/} of the Grassmannian 
$Gr(k,n,\C)$, whose strata are the realization spaces of {\it matroids.}
They correspond to {\it matroid polytopes\/} living inside $\Delta_{kn}$.
In general, matroid strata are not cells.  In fact, according to Mn\"ev's universality 
theorem \cite{Mnev}, the matroid strata can be as complicated as any 
algebraic variety.  Thus the matroid stratification of the Grassmannian 
can have arbitrarily bad behavior.

There is, however, a semialgebraic 
subset of the real Grassmannian $Gr(k,n,\R)$, called the
{\it nonnegative Grassmannian} $Gr^{\geq 0}(k,n)$, 
where the matroid stratification exhibits a well
behaved combinatorial and geometric structure.  Its structure, which is 
quite rich and nontrivial, can nevertheless be described in explicit terms.  
In some way, the nonnegative Grassmannian is similar to a polytope.

The notion of a {\it totally positive matrix,} that is a matrix
with all positive minors, originated in pioneering works of Gantmacher, Krein
\cite{GK35}, and Schoenberg \cite{Sch30}.
Since then such matrices appeared in many areas of pure and applied mathematics.
Lusztig \cite{Lus94, Lus2_98, Lus1_98} 
generalized the theory of total
positivity in the general context of Lie theory.  He defined the positive part
for a reductive Lie group $G$ and a generalized flag variety $G/P$.  
Rietsch \cite{Rie1, Rie2} studied its cellular decomposition.
Lusztig's theory of total positivity has close links with his theory of canonical 
bases \cite{Lus90, Lus92, Lus93}
and Fomin-Zelevinsky's cluster algebras \cite{FZ3, FZ2, FZ4, BFZ, FZ5}.

\cite{Pos2} initiated a combinatorial approach to the study of the
positive Grassmannian.  The positive
(resp., nonnegative) Grassmannian $Gr^{>0}(k,n)$
($Gr^{\geq 0}(k,n)$) 
was described as the subset of the Grassmannian $Gr(k,n,\R)$ where
all Pl\"ucker coordinates are positive 
(resp., nonnegative).
This ``elementary'' definition agrees with  Lusztig's general notion
\cite{Lus2_98} of the positive part of $G/P$ in the case when $G/P=Gr(k,n)$.

The {\it positroid cells,} defined as the parts of matroid strata 
inside the nonnegative Grassmannian, turned out to be indeed cells. 
(The term ``positroid'' is an abbreviation for ``positive matroid.'')
The positroid cells form a CW-complex.
Conjecturally, it is a {\it regular\/} CW-complex, and the closure
of each positroid cell is homeomorphic to a closed ball.
This positroid stratification of $Gr^{\geq 0}(k,n)$ 
is the common refinement of $n$ {\it cyclically shifted\/} Schubert decompositions
\cite{Pos2}.
Compare this with the result \cite{GGMS} that the matroid stratification of $Gr(k,n)$
is the common refinement of $n!$ {\it permuted\/} Schubert decompositions.
The {\it cyclic shift\/} 
plays a crucial role in the study of the positive Grassmannian.
Many objects associated with the positive Grassmannian exhibit 
{\it cyclic symmetry.}

Positroid cells were identified in \cite{Pos2} with many combinatorial
objects, such as decorated permutation, Grassmann necklaces, etc.
Moreover, an explicit birational subtraction-free 
parametrization of each cell was described
in terms of {\it plabic graphs,} that is, {\it planar bicolored\/} graphs,
which are certain graphs embedded in a disk with vertices colored in two colors. 

Remarkably, the combinatorial structures that appeared in the study 
of the positive Grassmannian
also surfaced and played an important role 
in many different areas of mathematics and physics.
Scott \cite{Sco1, Sco2} and \cite{OPS} linked
these objects with {\it cluster algebra\/} structure on the 
Grassmannian and with Leclerc-Zelevinsky's quasi-commutting families of 
{\it quantum minors\/} and {\it weakly separated collections.} 
Corteel and Williams \cite{CW} applied Le-diagrams (which correspond to positroids) 
to the study of the {\it partially asymmetric exclusion process\/}
(PASEP).  Knutson, Lam, and Speyer \cite{KLS} proved that 
the cohomology classes of the {\it positroid varieties\/}  
(the complexifications of the positroid cells) 
are given by the {\it affine Stanley symmetric functions,}
which are dual to Lapointe-Lascoux-Morse {\it $k$-Schur functions.}
They also linked positroids with {\it theory of juggling.}
Plabic graphs appeared in works of 
Chakravarty, Kodama, and Williams \cite{ChK, KoW11, KoW12}
as {\it soliton solutions\/} of the Kadomtsev-Petviashvili (KP) equation,
which describes nonlinear waves.
Last but not least, plabic graphs appeared under the name of {\it on-shell diagrams}
in the work by Arkani-Hamed et al \cite{ABCGPT} on {\it scattering amplitudes\/}
in $\mathcal{N}=4$ supersymmetric Yang-Mills (SYM) theory.  
They play a role somewhat similar to Feynman diagrams,
however, unlike Feynman diagrams, they represent on-shell processes
and do not require introduction of virtual particles.

\medskip

In this paper, we review some of the main constructions and results from
\cite{Pos2, PSW, OPS} related to the positive Grassmannian.  
We extend these constructions in the language of {\it Grassmannian
graphs.}  The parametrization of a positroid cell in $Gr^{\geq 0}(k,n)$ 
given by a Grassmannian graph can be thought of as a way to ``glue'' 
the positroid cell out of ``little positive Grassmannians'' associated with 
vertices of the graph.  The idea to think about parametrizations 
of cells as gluings of Grassmannians came originally from physics \cite{ABCGPT},
where vertices of on-shell diagrams (i.e., plabic graphs) were viewed as 
little Grassmannians $Gr(1,3)$ and $Gr(2,3)$.

We link this construction of parametrizations of $Gr^{>0}(k,n)$ given by
Grassmannian graphs with the study of polyhedral subdivisions induced by
$2$-dimensional cyclic projections $\pi:\Delta_{kn}\to Q$ of the hypersimplex.
Reduced Grassmannian graphs parametrizing the positive Grassmannian $Gr^{>0}(k,n)$ 
turn out to be in bijection with $\pi$-induced polyhedral subdivisions.
Thus gluing of Grassmannians from smaller Grassmannians is equivalent to
subdividing polytopes into smaller polytopes.  
The study of $\pi$-induced subdivisions for projections of polytopes 
is the subject of Billera-Sturmfels' theory \cite{BS} of {\it fiber polytopes\/} 
and the {\it generalized Baues problem\/} (GBP) posed by 
Billera, Kapranov, Sturmfels \cite{BKS}.
We also mention the result of Galashin~\cite{Gal} where plabic
graphs are identified with sections of {\it zonotopal tilings,} 
and the construction from the joint work \cite{LP} with Lam on {\it polypositroids\/}
where plabic graphs are viewed as {\it membranes,} 
which are certain 2-dimensional surfaces in higher
dimensional spaces. 

The correspondence between parametrizations 
of the positive Grassmannian and polyhedral subdivisions leads 
to natural generalizations and conjectures.
We discuss a possible extension of constructions 
of this paper to ``higher positive Grassmannians'' and amplituhedra 
of Arkani-Hamed and Trnka \cite{AT}.

\medskip

I thank 
Federico Ardila,
Nima Arkani-Hamed, 
Arkady Berenstein,
David Bernstein,
Lou Billera,
Jacob Bourjaily, 
Freddy Cachazo,
Miriam Farber, 
Sergey Fomin, 
Pavel Galashin, 
Israel Moiseevich Gelfand,
Oleg Gleizer,
Alexander Goncharov, 
Darij Grinberg,
Alberto Gr\"unbaum, 
Xuhua He, 
Sam Hopkins,
David Ingerman, 
Tam\'as K\'alm\'an,
Mikhail Kapranov,
Askold Khovanskii,
Anatol Kirillov,
Allen Knutson, 
Gleb Koshevoy,
Thomas Lam, 
Joel Lewis,
Gaku Liu,
Ricky Liu,
George Lusztig, 
Thomas McConville,
Karola M\'esz\'aros,
Alejandro Morales,
Gleb Nenashev,
Suho Oh, 
Jim Propp, 
Pavlo Pylyavskyy,
Vic Reiner,
Vladimir Retakh,
Konni Rietsch, 
Tom Roby,
Yuval Roichman,
Paco Santos,
Jeanne Scott, 
Boris Shapiro,
Michael Shapiro, 
David Speyer,
Richard Stanley, 
Bernd Sturmfels, 
Dylan Thurston, 
Jaroslav Trnka, 
Wuttisak Trongsiriwat,
Vladimir Voevodsky,
Lauren Williams,
Hwanchul Yoo,
Andrei Zelevinsky,
and G\"unter Ziegler
for insightful  conversations.
These people made a tremendous contribution to the study of the
positive Grassmannian and related combinatorial, algebraic, geometric, topological, 
and physical structures.  
Many themes we discuss here are from past and future
projects with various subsets of these people.

\section{Grassmannian and matroids}

Fix integers $0\leq k\leq n$. Let $[n]:=\{1,\dots,n\}$ and 
$[n]\choose k$ be the set of $k$-element subsets of $[n]$.

The {\it Grassmannian\/} $Gr(k,n)=Gr(k,n,\F)$ over a field $\F$
is the variety of $k$-dimensional linear subspaces in $\F^n$. 
More concretely, $Gr(k,n)$ is the space of 
$k\times n$-matrices of rank $k$ modulo the left action of $GL(k)=GL(k,\F)$.
Let $[A]=GL(k)\, A$ be the element of $Gr(k,n)$ 
represented by matrix $A$.

Maximal minors $\Delta_I(A)$ of such matrices $A$, 
where $I\in {[n]\choose k}$, form projective coordinates on 
$Gr(k,n)$, called the {\it Pl\"ucker coordinates}.
For $[A]\in Gr(k,n)$, let 
$$
\M(A):=\{I\in {[n]\choose k}\mid \Delta_I(A)\ne 0\}.
$$
The sets of the form $\M(A)$ are a special kind of {\it matroids,}
called {\it $\F$-realizable matroids.} 
{\it Matroid strata\/} are the realization spaces 
of realizable matroids $\M\subset {[n]\choose k}$: 
$$
S_\M:=\{[A]\in Gr(k,n)\mid \M(A)=\M\}.
$$
The {\it matroid stratification\/} is the disjoint decomposition
$$
Gr(k,n) = \bigsqcup_{\M \textrm{ realizable matroid}} S_\M\,.
$$
The {\it Gale order\/} ``$\preceq$'' (or the coordinatewise order) 
is the partial order on $[n]\choose k$  given by
$
\{i_1<\dots <i_k\}\preceq\{j_1<\cdots<j_k\},
\textrm{ if }i_r\leq j_r \textrm{ for }r\in[k].
$
Each matroid $\M$ has a unique minimal element
$I_{\min}(\M)$ with respect to the Gale order.

For $I\in{[n]\choose k}$, the {\it Schubert cell\/} $\Omega_I\subset Gr(k,n)$
is given by
$$
\Omega_I := \{[A]\in Gr(k,n)\mid I = I_{\min}(\M(A)\} = 
\bigsqcup_{\M:\, I=I_{\min}(\M)} S_\M.
$$
They form the {\it Schubert decomposition\/} $Gr(k,n)=\bigsqcup \Omega_{I}$.
Clearly, 
for a realizable matroid $\M$, we have 
$S_\M\subset \Omega_I$ if and only if $I=I_{\min}(\M)$.

The {\it symmetric group\/} $S_n$ acts on $Gr(k,n)$ by permutations 
$w([v_1,\dots,v_n])=[v_{w(1)},\dots,v_{w(n)}]$
of columns of $[A]=[v_1,\dots,v_n]\in Gr(k,n)$.

It is clear that, see \cite{GGMS}, the matroid stratification of $Gr(k,n)$ 
is the common refinement of the $n!$ permuted Schubert decompositions. 
In other words, each matroid stratum $S_\M$ is an intersection of
permuted Schubert cells:
$$
S_\M =\bigcap_{w\in S_n} w(\Omega_{I_w}).
$$
Indeed, if we know the minimal elements of a set $\M\subset {[n]\choose k}$
with respect to all $n!$ orderings of $[n]$, we know the set $\M$ itself.

\section{Positive Grassmannian and positroids}

Fix the field $\F=\R$.  Let $Gr(k,n)=Gr(k,n,\R)$
be the real Grassmannian.

\begin{definition} {\rm \cite[Definition~3.1]{Pos2}}  \
The {\it positive Grassmannian\/} $Gr^{>0}(k,n)$ (resp., 
{\it nonnegative Grassmannian} $Gr^{\geq 0}(k,n)$)
is the semialgebraic set of elements $[A]\in Gr(k,n)$ 
represented by $k\times n$ matrices $A$ with all positive 
maximal minors $\Delta_I(A)>0$ 
(resp., all nonnegative maximal minors $\Delta_I(A)\geq 0$). 
\end{definition}

This definition agrees with Lusztig's general definition \cite{Lus2_98}
of the positive part of a generalized flag variety $G/P$ in the case 
when $G/P=Gr(k,n)$.

\begin{definition} {\rm \cite[Definition~3.2]{Pos2}} \ 
A {\it positroid cell\/} $\Pi_\M \subset Gr^{\geq 0}(k,n)$ is 
a nonempty intersection of a matroid stratum with the nonnegative Grassmannian:
$$
\Pi_\M := S_\M\cap Gr^{\geq 0}(k,n).
$$
A {\it positroid\/} of {\it rank\/} $k$ is 
a collection $\M\subset {[n]\choose k}$ such that 
$\Pi_\M$ is nonempty.
The {\it positroid stratification\/} of the nonnegative Grassmannian is the disjoint
decomposition of $Gr^{\geq 0}(k,n)$ into the positroid cells:
$$
Gr^{\geq 0}(k,n) = \bigsqcup_{\M\textrm{ is a positroid}} \Pi_\M.
$$
\end{definition}

Clearly, positroids, or positive matroids, are a special kind of matroids. 
The positive Grassmannian $Gr^{>0}(k,n)$ itself is the 
{\it top positroid cell\/} $\Pi_{{[n]\choose k}}$ 
for the uniform matroid $\M= {[n]\choose k}$.

The {\it cyclic shift\/} is the map $\tilde c:Gr(k,n)\to Gr(k,n)$ acting
on elements $[A]=[v_1,\dots,v_n]\in Gr(k,n)$ by
$$
\tilde c:[v_1,\dots,v_n]
\longmapsto
[v_2,v_3,\dots,v_n,(-1)^{k-1}v_1].
$$ 
The shift $\tilde c$ induces the action of the
{\it cyclic group\/} $\Z/n\Z$ on the Grassmannian $Gr(k,n)$,
that preserves its positive part $Gr^{>0}(k,n)$.
Many of the objects associated with the positive Grassmannian
exhibit {\it cyclic symmetry.}  This cyclic symmetry is a crucial
ingredient in the study of the positive Grassmannian.

\begin{theorem}
\label{th:Pishifts} 
\cite[Theorem~3.7]{Pos2}  
The positroid stratification
is the common refinement of $n$ {\it cyclically shifted\/} Schubert
decompositions restricted to $Gr^{\geq 0}(k,n)$.
In other words, each positroid cell\/ $\Pi_\M$ is given by the intersection
of the nonnegative parts of $n$ cyclically shifted Schubert cells:
$$
\Pi_\M = \bigcap_{i=0}^{n-1}  \tilde{c}^{\,i}(\Omega_{I_i} \cap Gr^{\geq 0}(k,n)).
$$
\end{theorem}

So the positroid cells require intersecting $n$ cyclically shifted 
Schubert cells, which is a smaller number than $n!$ 
permuted Schubert cells needed for general matroid strata.
In fact, the positroid cells $\Pi_\M$ (unlike matroid strata) are indeed cells.

\begin{theorem} {\it \cite[Theorem~3.5]{Pos2}, \cite[Theorem~5.4]{PSW}} \
The positroid cells $\Pi_\M$ are homeomorphic to open balls.
The cell decomposition of $Gr^{\geq 0}(k,n)$ into the positroid cells
$\Pi_\M$ is a CW-complex.
\end{theorem}

\begin{conjecture} \cite[Conjecture~3.6]{Pos2} The positroid stratification of
the nonnegative Grassmannian $Gr^{\geq 0}(k,n)$ is a regular CW-complex.  In
particular, the closure $\overline{\Pi}_\M$ of each positroid cell in $Gr^{\geq
0}(k,n)$ is homeomorphic to a closed ball.  \end{conjecture}

This conjecture was motivated by a similar conjecture of Fomin and Zelevinsky on
double Bruhat cells~\cite{FZ1}.
Up to homotopy-equivalence this conjecture was proved 
by Rietsch and Williams \cite{RW}.  A major step towards this conjecture 
was recently achieved by Galashin, Karp, and Lam,
who proved it for the top cell.

\begin{theorem} {\cite[Theorem~1.1]{GKL}} \   
The nonnegative Grassmannian $Gr^{\geq 0}(k,n)$ is homeomorphic to a 
closed ball of dimension $k(n-k)$.
\end{theorem}

By Theorem~\ref{th:Pishifts}, 
positroids $\M$ and
positroid cells $\Pi_\M\subset Gr^{\geq 0}(k,n)$
correspond to certain sequences
$(I_0,I_1,\dots,I_{n-1})$.
Let us describe this bijection explicitly.

\begin{definition}
\cite[Definition~16.1]{Pos2}
A {\it Grassmann necklace\/} $\J=(J_1,J_2,\dots,J_{n})$ 
of type $(k,n)$ is a sequence 
of elements $J_i\in {[n]\choose k}$ such that, for any $i\in [n]$,
either $J_{i+1} = (J_i\setminus \{i\})\cup \{j\}$ or $J_{i+1}=J_i$,
where the indices $i$ are taken ${\pmod n}$.
\end{definition}

The {\it cyclic permutation\/} $c\in S_n$ is given by $c:i\mapsto i+1\pmod n$.
The action of the symmetric group $S_n$ on $[n]$ induces the $S_n$-action 
on ${[n]\choose k}$ and on subsets of ${[n]\choose k}$.
Recall that $I_{\min}(\M)$ is the minimal element of a matroid $\M$
in the Gale order. 
For a matroid $\M$, let
$$
\begin{array}{l}
\J(\M):=(J_1,\dots,J_n), \textrm{ where }\\[.1in]
J_{i+1} = c^{i}(I_{\min}(c^{-i}(\M))), \textrm{ for } i=0,\dots,n-1.
\end{array}
$$

\begin{theorem}
\cite[Theorem~17.1]{Pos2}
The map $\M\mapsto \J(\M)$ is a bijection between
positroids $\M$ of rank $k$ on the ground set $[n]$
and Grassmann necklaces of type $(k,n)$.
\end{theorem}

The sequence $(I_0,I_1,\dots,I_{n-1})$ associated with 
$\M$ as in Theorem~\ref{th:Pishifts} is related to the Grassmann necklace
$(J_1,\dots,J_n)$ of $\M$ by $I_i=c^{-i}(J_{i+1})$, for $i=0,\dots,n-1$.

The following result shows how to reconstruct a positroid 
$\M$ from its Grassmann necklace, cf.~Theorem~\ref{th:Pishifts}.
For $I\in {[n]\choose k}$, the {\it Schubert matroid\/} is
$\mathcal{M}_{I}:=\{J\in{[n]\choose k}\mid I\preceq J\}$, where ``$\preceq$'' is the Gale order.

\begin{theorem}
\label{th:M=Schubert}
\cite[Theorem~6]{Oh}
For a Grassmann necklace $\J=(J_1,\dots,J_n)$, 
the associated positroid $\M(\J)=\M$ is given by
$$
\M=\bigcap_{i=0}^{n-1} c^i (\mathcal{M}_{I_i}),
$$
where $I_i = c^{-i}(J_{i+1})$.
\end{theorem}

Let us describe positroids in the language of convex geometry.
The {\it hypersimplex\/} 
$$
\Delta_{kn} :=\conv\left\{e_I\mid I\in{[n]\choose k}\right\}
$$
is the convex hull of the $n\choose k$ points 
$e_I = \sum_{i\in I} e_i$, for all $I\in{[n]\choose k}$.
Here $e_1,\dots, e_n$ is the standard basis in $\R^n$.
For a subset $\M\subset{[n]\choose k}$, let $P_\M:=\conv\{e_I\mid I\in \M\}$
be the convex hull of vertices of $\Delta_{kn}$ associated with elements of $\M$.

By \cite{GGMS}, $\M$ is a matroid if and only if every edge of the 
polytope $P_\M$ has the form $[e_I,e_J]$, for $I,J\in{[n]\choose k}$ with 
$|I\cap J| = k-1$.  Here is an analogous description of positroids, 
which is not hard to derive from Theorem~\ref{th:M=Schubert}.

\begin{theorem} \cite{LP}
A nonempty subset $\M\subset{[n]\choose k}$ is a positroid if and only if
\begin{enumerate}
\item Every edge of $P_\M$ has the form $[e_I,e_J]$, 
for $I,J\in{[n]\choose k}$ with $|I\cap J|=k-1$.
\item Every facet of $P_\M$ is given by 
$x_i+x_{i+1}+\cdots + x_j = a_{ij}$ for some cyclic interval 
$\{i,i+1,\dots,j\}\subset [n]$ and $a_{ij}\in\Z$.
\end{enumerate}
\end{theorem}

Many of the results on the positive Grassmannian are based on an explicit
birational parametrization \cite{Pos2} 
of the positroid cells $\Pi_\M$ in terms of plabic graphs.  
In the next section we describe a more general class
of Grassmannian graphs that includes plabic graphs.

\section{Grassmannian graphs}

\begin{definition} 
A {\it Grassmannian graph\/} is a finite graph $G=(V,E)$,
with vertex set $V$ and edge set $E$, embedded into a disk 
(and considered up to homeomorphism) with $n$ 
{\it boundary vertices\/} $b_1,\dots,b_n\in V$ 
of degree 1 on the boundary of the disk (in the clockwise order), 
and possibly some {\it internal vertices\/} $v$ in the interior of the disk
equipped with integer parameters $h(v)\in\{0,1,\dots,\deg(v)\}$,
called {\it helicities\/} of vertices. 
Here $\deg(v)$ is the degree of vertex $v$.
We say that an internal vertex $v$ is of {\it type\/}
$(h,d)$ if $d=\deg(v)$ and $h=h(v)$.

The set of internal vertices of $G$ is denoted by
$V_\inte=V\setminus\{b_1,\dots,b_n\}$, and the set of {\it internal edges,}
i.e., the edges which are not adjacent to the boundary vertices, 
is denoted by $E_\inte\subset E$.
The {\it internal subgraph\/} is $G_\inte=(V_\inte,E_\inte)$.
\smallskip

A {\it perfect orientation\/} of a Grassmannian graph $G$ is a choice 
of directions for all edges $e\in E$ of the graph $G$ such that,
for each internal vertex $v\in V_\inte$ with helicity $h(v)$, exactly 
$h(v)$ of the edges adjacent to $v$ are directed towards $v$ and 
the remaining $\deg(v)-h(v)$ of adjacent edges are directed away from $v$.
A Grassmannian graph is called {\it perfectly orientable\/}
if it has a perfect orientation.

\smallskip

The {\it helicity\/} of a Grassmannian graph $G$ 
with $n$ boundary vertices is the number $h(G)$ 
given by 
$$
h(G) - n/2 = \sum_{v\in V_\inte} (h(v)  - \deg(v)/2). 
$$
\end{definition}

For a perfect orientation $\O$ of $G$, let 
$I(\O)$ be the set of indices $i\in[n]$ such that the boundary edge adjacent 
to $b_i$ is directed towards the interior of $G$
in the orientation $\O$.

\begin{lemma}
For a perfectly orientable Grassmannian graph $G$ and any 
perfect orientation $\O$ of $G$, we have $|I(\O)| = h(G)$. 
In particular, in this case, $h(G)\in\{0,1,\dots,n\}$.
\end{lemma}

\begin{remark}
This lemma expresses the {\it Helicity Conservation Law.}
We leave it as an exercise for the reader.
\end{remark}

For a perfectly orientable Grassmannian graph $G$ of helicity 
$h(G)=k$, let 
$$
\M(G)=\{I(\O)\mid \O\textrm{ is a perfect orientation of }G\}\subset 
{[n]\choose k}.
$$

Here is one result that links Grassmannian graphs with positroids.

\begin{theorem}
\label{th:M=M(G)}
For a perfectly orientable Grassmannian graph $G$ with
$h(G)=k$, the set $\M(G)$ is a positroid of rank $k$.
All positroids have form $\M(G)$ for some $G$.
\end{theorem}

\begin{definition}
A {\it strand\/} $\alpha$ in a Grassmannian graph $G$ is a directed 
walk along edges of $G$ that either starts and ends 
at some boundary vertices, or is a closed 
walk in the internal subgraph $G_\inte$, 
satisfying the following {\it Rules of the Road:}
For each internal vertex $v\in V_\inte$ with adjacent edges labelled
$a_1,\dots,a_d$ in the clockwise order,
where $d=\deg(v)$, if $\alpha$ enters $v$ through the edge $a_i$, it leaves $v$ 
through the edge $a_{j}$, where $j= i+h(v) \pmod d$.
\smallskip

A Grassmannian  graph $G$ is {\it reduced\/} if 
\begin{enumerate}
\item There are no strands which are closed loops 
in the internal subgraph $G_\inte$.
\item All strands in $G$ are simple curves without self-intersections.
The only exception is that we allow strands
$b_i\to v\to b_i$ where $v\in V_\inte$ is a 
{\it boundary leaf,} that is a vertex of degree 1
connected with $b_i$ by an edge.
\item
Any two strands $\alpha\ne \beta$ cannot have a {\it bad double crossing,}
that is, a pair of vertices $u\ne v$ such that both $\alpha$ and $\beta$ 
pass through $u$ and $v$ and both are directed from $u$ to $v$.
(We allow double crossings where $\alpha$ goes from $u$ to $v$
and  $\beta$ goes from $v$ to $u$.)
\item 
The graph $G$ has no vertices of degree $2$.
\end{enumerate}
\smallskip

The {\it decorated strand permutation\/} $w=w_G$ of a reduced 
Grassmannian graph $G$ is the permutation $w:[n]\to[n]$
with fixed points colored in colors $0$ or $1$
such that 
\begin{enumerate}
\item
$w(i) = j$ if the strand that starts 
at the boundary vertex $b_i$ ends at the boundary vertex $b_j$.
\item
For a boundary leaf $v$ connected to $b_i$,
the decorated permutation $w$ has fixed point $w(i)=i$ 
colored in color $h(v)\in\{0,1\}$.
\end{enumerate}
\smallskip

A {\it complete\/} reduced Grassmannian 
graph $G$ of {\it type\/} $(k,n)$, for $0\leq k\leq n$,
is a reduced Grassmannian graph whose decorated strand permutation
is given by $w(i)=i+k\pmod n$.  
In addition, for $k=0$ (resp., for $k=n$),
we require that $G$ only has $n$ boundary 
leaves of helicity $0$ (resp., of helicity $1$)
and no other internal vertices.
\end{definition}

\begin{theorem}
\label{th:perf_orient_exists}
cf.\ \cite[Corollaries~14.7 and~14.10]{Pos2}
{\rm (1)}
For any permutation $w:[n]\to[n]$ 
with fixed points colored in $0$ or $1$, there
exists a reduced Grassmannian graph $G$ 
whose decorated strand permutation $w_G$ is $w$.

\smallskip

{\rm (2)}
Any reduced Grassmannian graph 
is perfectly orientable.  Moreover,
it has an acyclic perfect orientation.

\smallskip
{\rm (3)}
A reduced Grassmannian graph $G$ is complete of type $(k,n)$ 
if and only if its helicity equals $h(G)=k$ and the number of 
internal faces (excluding $n$ boundary faces) equals
$$
f(k,n)- \sum_{v\in V_\inte} f(h(v),\deg(v)).
$$
where $f(k,n) = (k-1)(n-k-1)$.
A reduced Grassmannian graph is complete if and only if it is not a proper
induced subgraph of a larger reduced Grassmannian graph.  
\end{theorem}

\begin{figure}[h]
\label{fig:Grass_graph}
\includegraphics[height=.8in,width=1.2in]{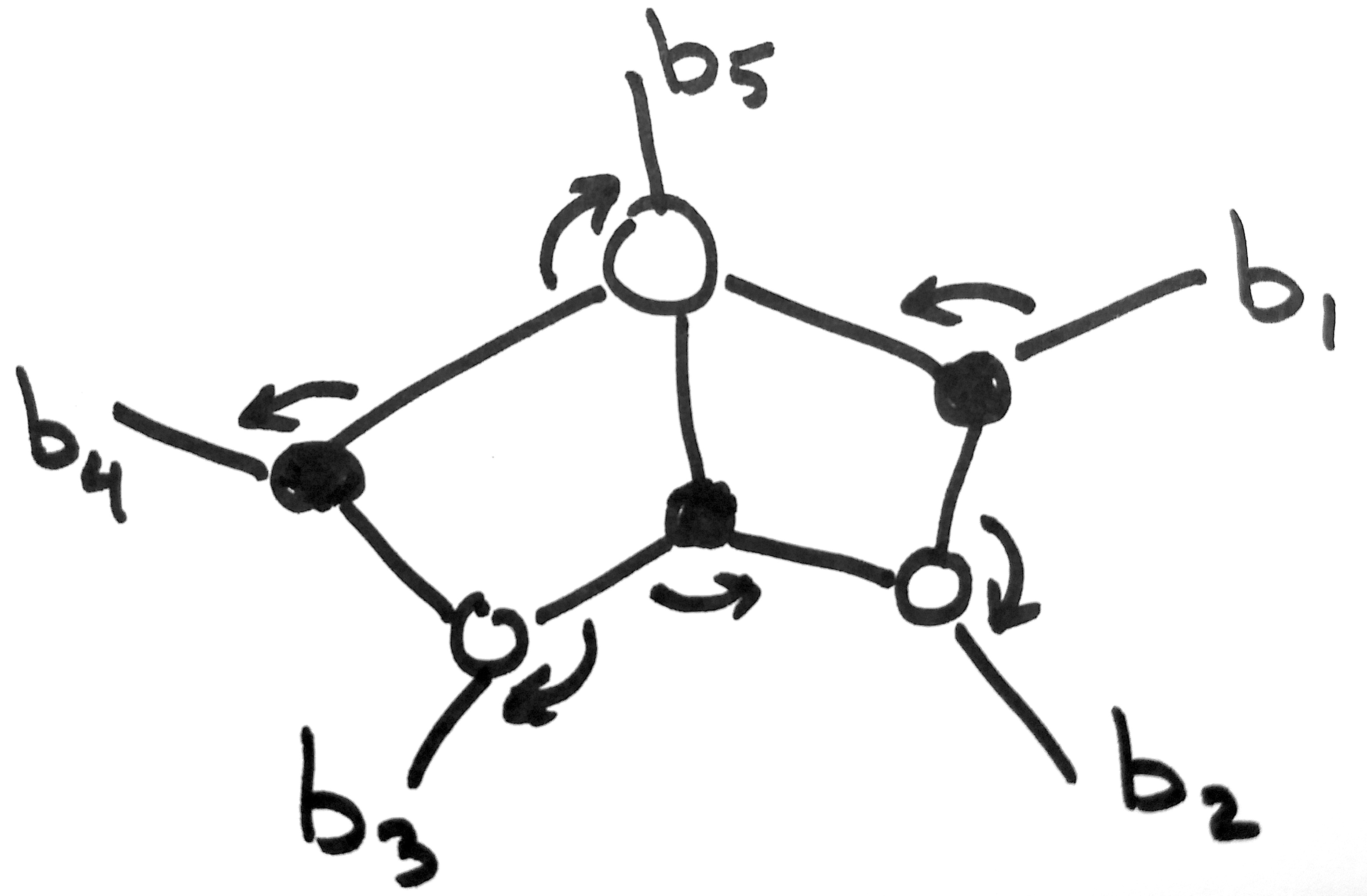}
\qquad \qquad
\includegraphics[height=.8in,width=1in]{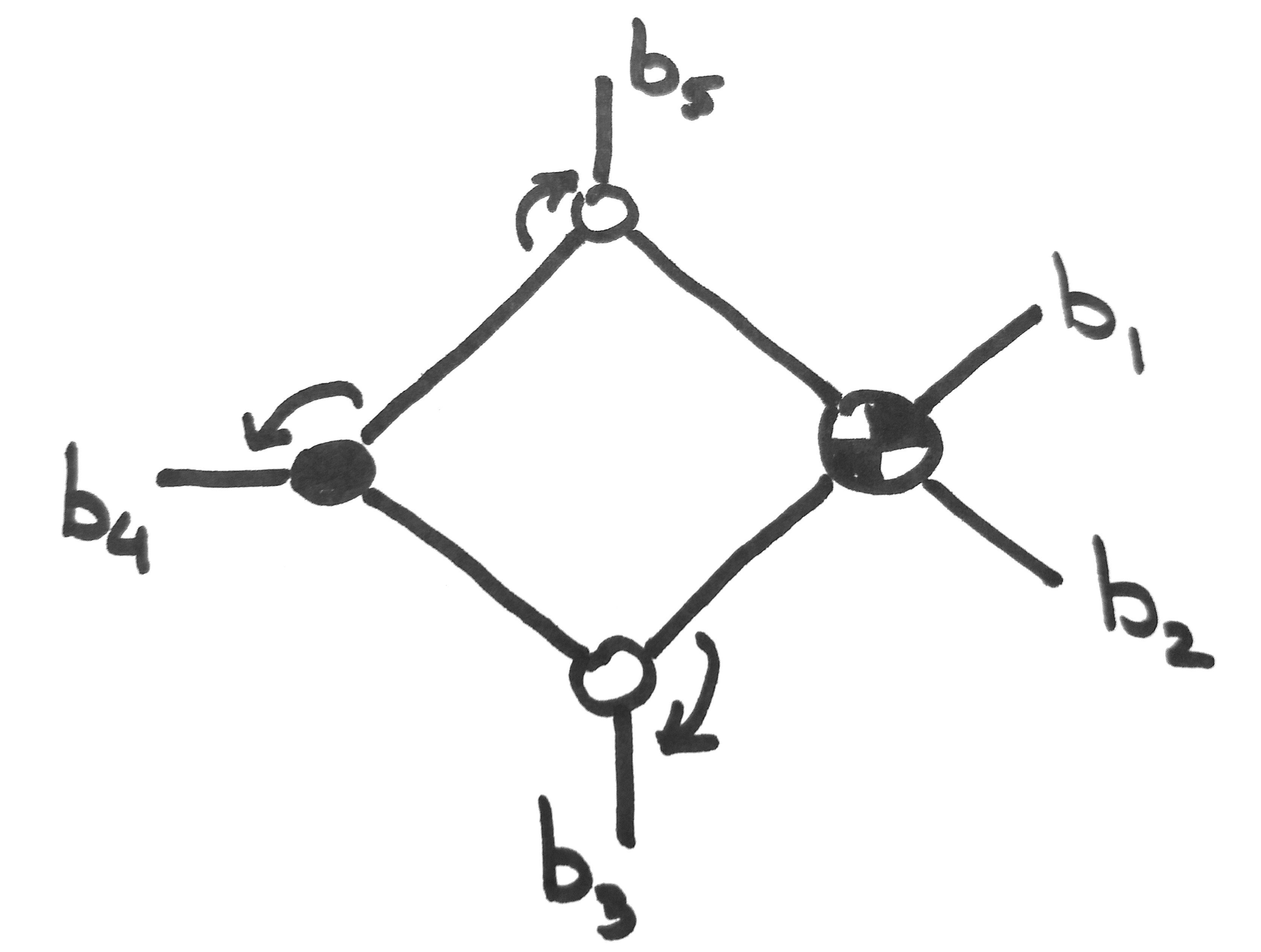}
\caption{Two complete reduced Grassmannian graphs of type $(2,5)$
with 2 internal faces (left) and 1 internal face (right).
The internal vertices of types $(1,3)$ and $(1,4)$ 
are colored in white, the type $(2,3)$ vertices colored in black, 
and the type $(2,4)$ vertex is ``chessboard'' colored.}
\end{figure}

Let us now describe a partial ordering and an equivalence relation
on Grassmannian graphs.

\begin{definition}
For two Grassmannian graphs $G$ and $G'$, we say that $G$ {\it refines\/} $G'$
(and that $G'$ {\it coarsens\/} $G$), if $G$ can be obtained from $G'$
by a sequence of the following operations:
Replace an internal vertex of type $(h,d)$ by a complete 
reduced Grassmannian graph of type $(h,d)$.

The {\it refinement order\/} on Grassmannian graphs is the partial order
$G\leq_\refi G'$ if $G$ refines $G'$. 
We say that $G'$ {\it covers\/} $G$, if $G'$ covers $G$ in the 
refinement order.

Two Grassmannian graphs $G$ and $G'$ are 
{\it refinement-equivalent\/} 
if they are 
in the same connected component of the refinement order $\leq_\refi$,
that is,  they can be obtained from each other by a sequence
of refinements and coarsenings.
\end{definition}

\begin{definition}
A Grassmannian graph is called a {\it plabic graph\/}
if it is a minimal element in the refinement order.
\end{definition}

The following is clear.

\begin{lemma}
A Grassmannian graph is a plabic graph if and only if 
each internal vertex in the graph has type $(1,3)$, $(2,3)$,
$(0,1)$, or $(1,1)$.
\end{lemma}

In drawings of plabic and Grassmannian graphs, 
we color vertices of types $(1,d)$ in white
color, and vertices of types $(d-1,d)$ in black color.

Let us now describe almost minimal elements in the refinement order.

\begin{definition}
A Grassmannian graph $G$ is called {\it almost plabic\/}
if it covers a plabic graph (a minimal element) in the refinement order.
\end{definition}

For example, the two graphs shown on Figure~\ref{fig:Grass_graph} 
are almost plabic.
The following lemma is also straightforward from the definitions.

\begin{lemma}
Each almost plabic Grassmannian graph $G$ has 
exactly one internal vertex (special vertex) 
of type $(1,4)$, $(2,4)$, $(3,4)$, $(0,2)$, $(1,2)$, or $(2,2)$,
and all other internal vertices of types $(1,3)$, $(2,3)$, $(0,1)$, or $(1,1)$.
An almost plabic graph with a special vertex of type of type $(1,4)$, $(2,4)$, or
$(3,4)$ covers exactly two plabic graphs.
An almost plabic graph with a special vertex of type $(0,2)$,
$(1,2)$, or $(2,2)$ covers exactly one plabic graph.
\end{lemma}

Note that a reduced Grassmannian graph cannot contain any vertices
of degree 2.
So each reduced almost plabic graph covers exactly two reduced plabic graphs.

\begin{definition}
Two plabic graphs are connected by 
a {\it move\/} of {\it type\/} $(1,4)$, $(2,4)$, or $(3,4)$,
if they are both covered by an almost plabic graph
with a special vertex of the corresponding type.
Two plabic graphs $G$ and $G'$ are {\it move-equivalent\/}
if they can be obtained from each other by a sequence of such moves.
\end{definition}

\begin{figure}[h]
\label{fig:plabic_moves}
\includegraphics[height=.4in,width=1.3in]{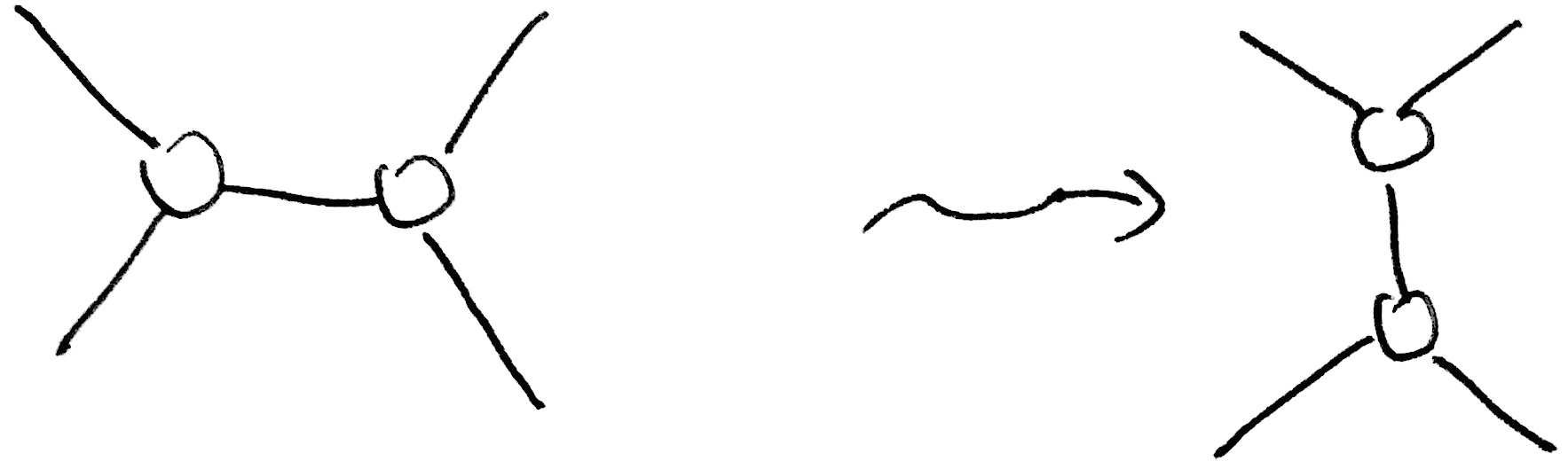}
\qquad
\quad
\includegraphics[height=.4in,width=1.3in]{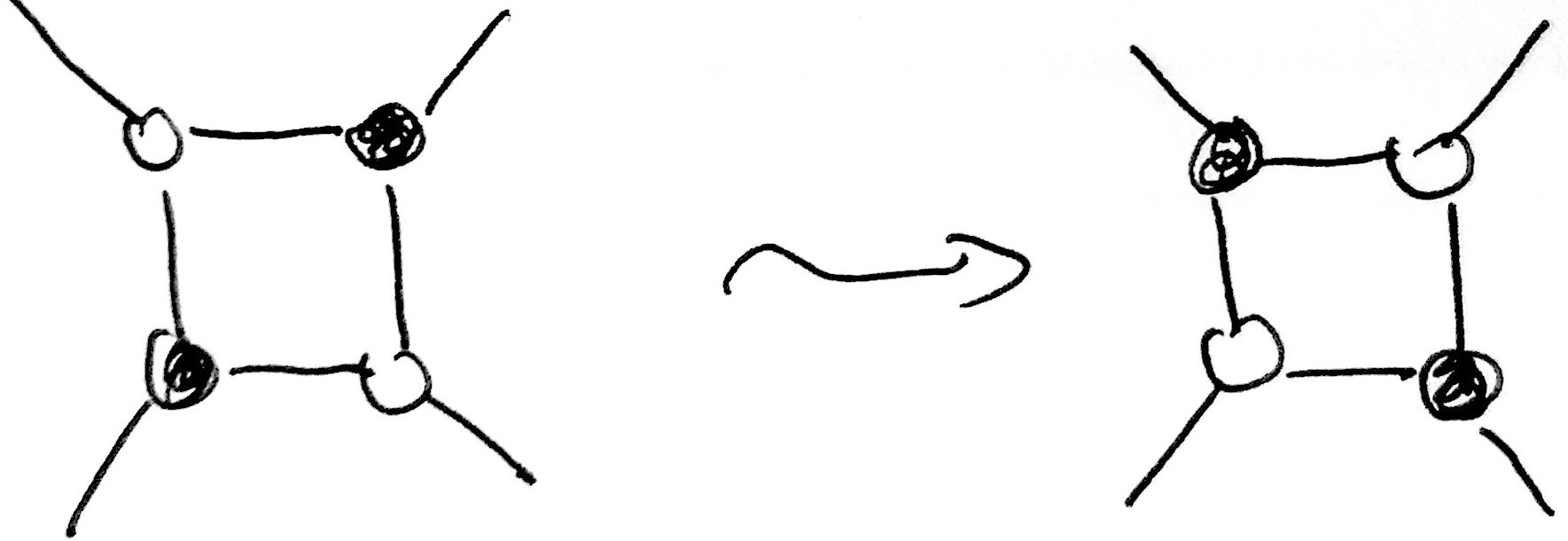}
\qquad
\quad
\includegraphics[height=.4in,width=1.3in]{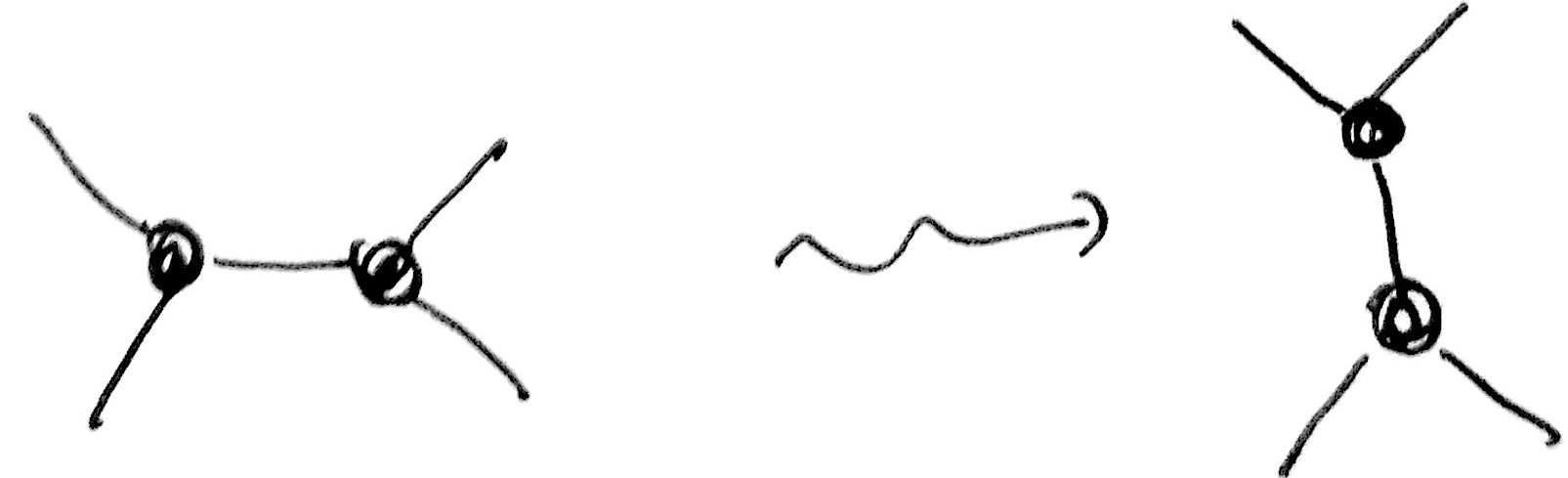}
\caption{Three types of moves of plabic graphs: (1,4) contraction-uncontraction of 
white vertices, (2,4) square move, (3,4) contraction-uncontraction of black
vertices.}
\end{figure}

Let us say that vertices of types $(1,2)$,
$(0,d)$, $(d,d)$ (except boundary leaves) are {\it extraneous.}
A reduced graph cannot have a vertex of this form.

\begin{theorem}
\label{th:move_refinement_equivalence}
{\rm (1)}
For two reduced Grassmannian graphs $G$ and $G'$,
the graphs are refinement-equivalent if and only 
if they have the same decorated strand permutation $w_{G}=w_{G'}$.

\smallskip

{\rm (2)} 
{\rm cf.~\cite[Theorem~13.4]{Pos2}} \
For two reduced plabic graphs $G$ and $G'$, 
the following are equivalent:
\begin{enumerate}
\item[(a)] The graphs are move-equivalent.
\item[(b)] The graphs are refinement-equivalent.
\item[(c)] The graphs have the same decorated 
     strand permutation $w_{G}=w_{G'}$.
\end{enumerate}

\smallskip

{\rm (3)}  A Grassmannian graph is reduced if and only if it 
has no extraneous vertices and is not 
refinement-equivalent to a graph with a pair of parallel
edges (two edges between the same vertices), 
or a loop-edge (an edge with both ends attached to the same vertex).

\smallskip

{\rm (4)} 
A plabic graph is reduced if and only 
if it has no extraneous vertices and is not move-equivalent 
to a plabic graph with a pair of parallel
edges or a loop-edge.
\end{theorem}

\begin{remark}
Plabic graphs are similar to {\it wiring diagrams\/} that 
represent decompositions of permutations into products of
adjacent transpositions.  
In fact, plabic graphs extend the notions of wiring diagrams
and, more generally, {\it double wiring diagrams\/} of Fomin-Zelevinsky \cite{FZ1},
see \cite[Remark~14.8, Figure~18.1]{Pos2}.
Moves of plabic graphs are analogous to Coxeter moves
of decompositions of permutations.
Reduced plabic graphs extend the notion of reduced decompositions 
of permutations.
\end{remark}

Let us now summarize the results about the relationship between 
positroids, Grassmannian and plabic graphs, decorated permutations,
and Grassmann necklaces.
For a decorated permutation $w:[n]\to[n]$ (a permutation with fixed
points colored $0$ or $1$), define
$\J(w):=(J_1,\dots,J_n)$, where 
$$
J_{i} = \{j \in [n] \mid c^{-i+1}w^{-1}(j)> c^{-i+1}(j)\} \cup \{j\in [n]\mid 
w(j)= j \textrm{ colored }1\}.
$$
The {\it helicity\/} of $w$ is defined as
$h(w) := |J_1|=\cdots=|J_n|$.
Conversely, for a Grassmann necklace $\J=(J_1,\dots,J_n)$,
let
$$
w(\J):=w, 
\quad\textrm{where }
w(i)=
\left\{
\begin{array}{l}
j \textrm{ if }J_{i+1} = (J_{i}\setminus\{i\})\cup\{j\},\\[.1in]
i \textrm{ (colored $0$) if }i\not\in J_{i} = J_{i+1},\\[.1in]
i \textrm{ (colored $1$) if }i\in J_i= J_{i+1}.
\end{array}
\right.
$$

\begin{theorem}
\label{th:positroids_perm_neckl_moves}
{\rm cf.~\cite{Pos2}} \
The following sets are in one-to-one correspondence:
\begin{enumerate}
\item Positroids $\M$ of rank $k$ on $n$ elements.
\item Decorated permutation $w$ of size $n$ and helicity $k$.
\item Grassmann necklaces $\J$ of type $(k,n)$.
\item Move-equivalence classes of reduced plabic graphs $G$ with $n$ boundary vertices
and helicity $h(G)=k$.
\item Refinement-equivalence classes of reduced Grassmannian graphs $G'$ with
$n$ boundary vertices and helicity $h(G')=k$.
\end{enumerate}
The following maps (described above in the paper) give explicit
bijection between these sets and form a commutative diagram:
\begin{enumerate}
\item Reduced Grassmannian/plabic graphs to positroids: $G\mapsto \M(G)$.
\item Reduced Grassmannian/plabic graphs to decorated permutations: $G\mapsto w_G$.
\item Positroids to Grassmann necklaces: 
$\M\mapsto \J(\M)$. 
\item Grassmann necklaces to positroids:  $\J\mapsto \M(\J)$,
\item Grassmann necklaces to decorated permutations: $\J\mapsto w(\J)$.
\item Decorated permutations to Grassmann necklaces: $w\mapsto \J(w)$.
\end{enumerate}
\end{theorem}

\begin{proof}[Proof of Theorems~\ref{th:M=M(G)}, 
\ref{th:perf_orient_exists},
\ref{th:move_refinement_equivalence},
\ref{th:positroids_perm_neckl_moves}]
In case of plabic graphs, most of these results were proved in \cite{Pos2}.
The extension of results to Grassmannian graphs  follows from 
a few easy observations.

Let $G$ and $G'$ be a pair of Grassmannian graphs such that $G$ refines $G'$.
Any perfect orientation of $G$ induces a perfect orientation of $G'$.
Conversely, any perfect orientation of $G'$ can be extended (not uniquely,
in general) to a perfect orientation of $G$.
Thus $G$ is perfectly orientable if and only if $G'$ is perfectly orientable,
and $\M(G)=\M(G')$ and $h(G)=h(G')$.  
Any strand of $G$ corresponds to a strand of $G'$.
The graph $G$ is reduced if and only if $G'$ is reduced.
If they are reduced, then they have the same decorated strand permutation
$w_G=w_{G'}$.
Finally, any Grassmannian graph can be refined to a plabic graph.
So the results for plabic graphs imply the results for Grassmannian graphs.
\end{proof}

\section{Weakly separated collections and cluster algebras}

\begin{definition} 
{\rm \cite{Sco1}, cf.~\cite{LZ}} \
Two subsets $I,J\in {[n]\choose k}$ are 
{\it weakly separated\/} if there is no $a<b<c<d$ such that
$a,c\in I\setminus J$ and $b,d\in J\setminus I$,
or vise versa.
A collection of subsets $S\subset {[n]\choose k}$ is 
{\it weakly separated\/}
if it is pairwise weakly separated.
\end{definition}

This is a variation of Leclerc-Zelevinsky's notion of weak separation
\cite{LZ} given by Scott \cite{Sco1}.
It appeared in their study of
quasi-commuting {\it quantum minors.} 

\begin{definition}
\label{def:face_label}
The {\it face labelling\/} of a reduced Grassmannian graph $G$
is the labelling of faces $F$ of $G$ by subsets $I_F\subset [n]$ given by 
the condition: For each strand $\alpha$ that goes from $b_i$ to $b_j$, 
we have $j\in I_F$
if and only if the face $F$ lies to the left of the strand $\alpha$
(with respect to the direction of the strand from $b_i$ to $b_j$).
\end{definition}

Let us stay that two reduced plabic graphs are {\it contraction-equivalent\/}
if they can be transformed to each other by the moves of type $(1,4)$ and $(3,4)$
(contraction-uncontraction moves) without using the move of type $(2,4)$ (square move).

\begin{theorem} 
\label{th:WS=plabic}
\cite{OPS}
{\rm (1)} 
Face labels of a reduced 
Grassmannian graph form a weakly separated collection
in ${[n]\choose k}$, where $k=h(G)$ is the helicity of $G$.

\smallskip

{\rm (2)}  Every maximal by inclusion weakly separated collection in
${[n]\choose k}$ is the collection of face labels of a complete 
reduced plabic graph of type $(k,n)$.  

\smallskip
{\rm (3)}  This gives a bijection between maximal by inclusion 
weakly separated collections in ${[n]\choose k}$, 
and contraction-equivalence classes of complete reduced plabic graphs
of type $(k,n)$.
\end{theorem}

\begin{remark}
Weakly separated collections are related to the {\it cluster algebra\/} structure
\cite{FZ3, FZ4, BFZ, FZ5} on the Grassmannian studied by Scott~\cite{Sco2}.
In general, the cluster algebra on $Gr(k,n)$ has infinitely
many clusters.  (See \cite{Sco2} for a classification of finite cases.)
There is, however, a nicely behaved finite set of clusters, called 
the {\it Pl\"ucker clusters,} which are formed by subsets of 
the Pl\"ucker coordinates $\Delta_I$.  
According to \cite[Theorem~1.6]{OPS}, the Pl\"ucker clusters 
for $Gr(k,n)$ are exactly the sets $\{\Delta_I\}_{I\in S}$ 
associated with maximal weakly-separated collections 
$S\subset {[n]\choose k}$. 
They are in bijection with contraction-equivalence classes
of type $(k,n)$ complete reduced plabic graphs, and are given by the $k(n-k)+1$
face labels of such graphs.
Square moves of plabic graphs correspond to {\it mutations\/}
of Pl\"ucker clusters in the cluster algebra.

Theorem~\ref{th:WS=plabic} implies an affirmative answer to 
the {\it purity conjecture\/} of Leclerc and Zelevinsky \cite{LZ}.
An independent solution of the purity conjecture was given 
by Danilov, Karzanov, and Koshevoy \cite{DKK} in terms of {\it generalized
tilings.}
The relationship between the parametrization a positroid cell given by a
plabic graph $G$ (see Section~\ref{sec:perfect_orientation_parametrization} below)
and the Pl\"ucker cluster $\{\Delta\}_{I\in S}$ associated with the same graph
$G$ induces a nontrivial transformation,
called the {\it twist map,} which was explicitly described by
Muller and Speyer \cite{MuS}.
Weakly separated collections appeared in the
study of {\it arrangements of equal minors\/} \cite{FP}.
In \cite{GP} the notion of weakly separated collections 
was extended in the general framework of {\it oriented matroids\/}
and {\it zonotopal tilings.}

\end{remark}

\section{Cyclically labelled Grassmannian}

Let us reformulate the definition of the Grassmannian 
and its positive part in a more
invariant form, which makes its cyclic symmetry manifest.  
In the next section, we will consider ``little positive Grassmannians'' 
associated with vertices $v$ of a Grassmannian graph $G$ whose ground sets
correspond to the edges adjacent to $v$.  There is no natural total ordering 
on such a set of edges, however there is the natural cyclic (clockwise) 
ordering.

We say that a {\it cyclic ordering\/} of a finite set $C$ 
is a choice of closed directed cycle that visits each element of $C$ 
exactly once.  A total ordering of $C$ is {\it compatible\/} with a cyclic ordering 
if it corresponds to a directed path on $C$ obtained by removing an edge of the cycle.
Clearly, there are $|C|$ such total orderings.

\begin{definition}
Let $C$ be a finite set of indices with a cyclic ordering of its elements,
and let $k$ be an integer between $0$ and $|C|$.
The {\it cyclically labelled Grassmannian\/} $Gr(k,C)$ over $\R$
is defined as the subvariety of the projective space $\P^{{|C|\choose k}-1}$ 
with projective Pl\"ucker coordinates $(\Delta_{I})$ labelled by 
{\it unordered\/} $k$-element subsets $I\subset C$
satisfying the Pl\"ucker relations 
written with respect to {\it any\/} total order ``$<$''
on $C$ compatible with the given cyclic ordering:
$$
\sum_{i\in A\setminus B}
(-1)^{ |\{a\in A,\, a>i\}| + |\{b\in B,\, b < i\}|}\,
\Delta_{A\setminus\{i\}}\,
\Delta_{B\cup\{i\}} = 0,
$$
where $A$ and $B$ are any $(k+1)$-element and $(k-1)$-element subsets of $C$,
respectively.
(More precisely, $Gr(k,C)$ is the projective 
algebraic variety given by the radical of the ideal 
generated by the above Pl\"ucker relations.)

The {\it positive part\/} $Gr^{>0}(k,C)$ is the subset of $Gr(k,C)$ 
where the Pl\"ucker coordinates can be simultaneously rescaled so that
$\Delta_I>0$, for all $k$-element subsets $I\subset C$.
\end{definition}

\begin{remark}
The Pl\"ucker relations (written as above) are invariant with respect to cyclic 
shifts of the ordering ``$<$''.
Thus the definition of the cyclically labelled Grassmannian $Gr(k,C)$ is 
independent of a choice of the total order on $C$. 
For example, for $k=2$ and $C=\{1,2,3,4\}$, $Gr(2,C)$ is 
the subvariety of $\P^{6-1}$ given by the Pl\"ucker relation:
$$
\Delta_{\{1,3\}} \,\Delta_{\{2,4\}} = 
\Delta_{\{1,2\}} \, \Delta_{\{3,4\}} + 
\Delta_{\{1,4\}} \, \Delta_{\{2,3\}}. 
$$
Observe the cyclic symmetry of this relation!  
The ordering of indices $2<3<4<1$ gives exactly the same $Gr(2,C)$ 
with the same positive part $Gr^{>0}(2,C)$.
\end{remark}

\begin{remark}
There is a subtle yet important difference between 
the cyclically labelled Grassmannian $Gr(k,C)$ 
with the  Pl\"ucker coorinates $\Delta_I$ 
and the  usual definition of the Grassmannian
$Gr(k,n)$, $n=|C|$, with the ``usual Pl\"ucker coordinates'' 
defined as the minors $D_{(i_1,\dots,i_k)} = \det(A_{i_1,\dots,i_k})$ 
of submatrices $A_{i_1,\dots,i_k}$ of a $k\times n$ matrix $A$.

The $D_{(i_1,\dots,i_k)}$ are labelled by {\it ordered\/} collections 
$(i_1,\dots,i_k)$ of indices.  They are {\it anti-symmetric\/}
with respect to permutations of the indices $i_1,\dots,i_k$.
On the other hand, the $\Delta_{\{i_1,\dots,i_k\}}$ are labelled
by {\it unordered\/} subsets $I=\{i_1,\dots,i_k\}$.  So they are 
{\it symmetric\/} with respect to permutations of the indices
$i_1,\dots,i_k$.

The ``usual Pl\"ucker relations'' for the $D_{(i_1,\dots,i_k)}$ have
the {\it $S_n$-symmetry\/} with respect to {\it all permutations\/} of the ground set.
On the other hand, the above Pl\"ucker relations for the $\Delta_{\{i_1,\dots,i_k\}}$
have only the {\it $\Z/n\Z$-symmetry\/} with respect to {\it cyclic shifts\/}
of the ground set.

Of course, if we fix a total order of the ground set, 
we can rearrange the indices in $D_{(i_1,\dots,i_k)}$ 
in the increasing order and identify 
$D_{(i_1,\dots,i_k)}$, for $i_1<\cdots<i_k$, with $\Delta_{\{i_1,\dots,i_k\}}$.
This identifies the cyclically labelled Grassmannian $Gr(k,C)$ 
with the usual Grassmannian $Gr(k,n)$.
However, this isomorphism is not canonical because it 
depends on a choice of the total ordering of the index set.  
For even $k$, the isomorphism 
is {\it not invariant\/} under cyclic shifts of the index set.
\end{remark}

\section{Perfect orientation parametrization of positroid cells}
\label{sec:perfect_orientation_parametrization}

Positroid cells were parametrized in \cite{Pos2} in terms of 
{\it boundary measurements\/} 
of perfect orientations of plabic graphs.  Equivalent descriptions 
of this parametrization were given in
terms of {\it network flows\/} by Talaska \cite{Talaska} and 
in terms of {\it perfect matchings\/} \cite{PSW, Lam}.
Another interpretation of this parametrization was motivated by
physics \cite{ABCGPT}, where plabic graphs were viewed as {\it on-shell
diagrams,} whose vertices represent little Grassmannians $Gr(1,3)$ and
$Gr(2,3)$ and edges correspond to gluings, see also \cite[Section 14]{Lam}
for a more mathematical description.
Here we give a simple and invariant way to describe the parametrization 
in the general setting of Grassmannian graphs and 
their perfect orientations.  It easily 
specializes to all the other descriptions.  Yet it clarifies the idea of
gluings of little Grassmannians.

\smallskip

Let $G=(V,E)$ be a perfectly orientable Grassmannian graph
with $n$ boundary vertices and helicity $h(G)=k$,
and let $G_\inte=(V_\inte,E_\inte)$ be its internal subgraph.
Also let $E_\bnd = E\setminus E_\inte$ be the set of boundary edges of $G$.

Informally speaking, each internal vertex $v\in V_\inte$ represents the
``little Grassmannian'' $Gr(h,d)$, where $d$ is the degree of vertex $v$ and $h$ is
its helicity.  We ``glue'' these little Grassmannians along the internal edges
$e\in E_\inte$ of the graph $G$ to form a subvariety in the ``big Grassmannian''
$Gr(k,n)$.  Gluing along each edge kills one parameter.  
Let us give a more rigorous description of this construction. 

For an internal vertex $v\in V_\inte$, 
let $E(v)\subset E$ be the set of all adjacent edges
to $v$ (possibly including some boundary edges),
which is cyclically ordered in the clockwise order (as we go
around $v$).
Define the {\it positive vertex-Grassmannian\/} $Gr^{>0}(v)$
as the positive part of the cyclically labelled Grassmannian
$$
Gr^{>0}(v):= Gr^{>0}(h(v),E(v)).
$$
Let $(\Delta_J^{(v)})$ be the Pl\"ucker  coordinates on $Gr^{>0}(v)$,
where $J$ ranges over the set ${E(v)\choose h(v)}$ of all 
$h(v)$-element subsets in $E(v)$.

Let us define several positive tori (i.e., positive parts of complex tori).
The {\it boundary positive torus\/} is $T_\bnd^{>0}:=(\R_{>0})^{E_\bnd}\simeq 
(\R_{>0})^n$.
The {\it internal positive torus\/} is $T_\inte^{>0}:=(\R_{>0})^{E_\inte}$,
and the {\it total positive torus\/} $T_\mathrm{tot}^{>0} := 
T_\bnd^{>0}\times T_\inte^{>0}$.
The boundary/internal/total positive torus is the group of $\R_{>0}$-valued
functions on boundary/internal/all edges of $G$.

These tori act on the positive vertex-Grassmannians $Gr^{>0}(v)$
by rescaling the Pl\"ucker coordinates.
For $(t_e)_{e\in E} \in T_\mathrm{tot}^{>0}$,
$$
(t_e):(\Delta_J^{(v)})\longmapsto
(\left(\prod_{e\in J } t_e\right)\, \Delta_J^{(v)}).
$$

The boundary torus $T_\bnd^{>0}$ also acts of the ``big Grassmannian'' $Gr(k,n)$
as usual $(t_1,\dots,t_n):\Delta_I\mapsto (\prod_{i\in I}t_i)\,\Delta_I$,
for $(t_1,\dots,t_n)\in T_\bnd^{>0}$.

Recall that, for a perfect orientation $\O$ of $G$,
$I(\O)$ denotes the set of $i\in[n]$ such that 
the boundary edge adjacent to $b_i$ is
directed towards the interior of $G$ in $\O$. 
For an internal vertex $v\in V_\inte$,
let $J(v,\O)\subset E(v)$ be the subset of edges adjacent to $v$
which are directed towards $v$ in the orientation $\O$.

We are now ready to describe the
{\it perfect orientation parametrization\/} of 
the positroid cells.

\begin{theorem}
Let $G$ be a perfectly orientable Grassmannian graph.
Let $\mu_G$ be the map defined on the direct product of the 
positive vertex-Grassmannians $Gr^{>0}(v)$
and written in terms of the Pl\"ucker coordinates as
$$
\begin{array}{l}
\displaystyle
\mu_G:\bigtimes_{v\in V_\inte} Gr^{>0}(v)\longrightarrow \P^{{[n]\choose k}-1}
\\[.2in]
\displaystyle
\mu_G: \bigtimes_{v\in V_\inte} (\Delta_J^{(v)})_{J\in {E(v)\choose h(v)}}
\longmapsto (\Delta_I)_{I\in {[n]\choose k}},
\end{array}
$$
where $\Delta_I$ is given by the sum over all 
perfect orientations $\O$ of the graph $G$
such that $I(\O)=I$: 
$$
\Delta_I = \sum_{I(\O)=I\  } 
\prod_{v\in V_\inte}  \Delta^{(v)}_{J(v,\O)}.
$$

{\rm (1)} The image of $\mu_G$ is exactly 
the positroid cell $\Pi_\M\subset Gr^{\geq 0}(k,n)\subset \P^{{[n]\choose k}-1}$,
where $\M=\M(G)$ is the positroid associated with $G$.

\smallskip

{\rm (2)}  The map $\mu_G$ is $T_\inte^{>0}$-invariant and $T_\bnd^{>0}$-equivariant,
that is, 
$\mu_G(t\cdot x) = \mu_G(x)$ for $t\in T_\inte^{>0}$, and
$\mu_G(t'\cdot x) = t'\cdot \mu_G(x)$ for $t'\in T_\bnd^{>0}$.

\smallskip

{\rm (3)}  
The map $\mu_G$  induces the birational 
subtraction-free bijection $\bar\mu_G$
$$
\bar \mu_G:
\left(\bigtimes_{v\in V_\inte} Gr^{>0}(v)\right)/T_\inte^{>0} \longrightarrow \Pi_G
$$
if and only if the Grassmannian graph  $G$ is a reduced.
\end{theorem}

\begin{remark}
The phrase ``birational subtraction-free bijection'' means
that both $\bar\mu_G$ and its inverse $(\bar\mu_G)^{-1}$ can be expressed 
in terms of the Pl\"ucker coordinates by rational (or even polynomial) 
expressions written without using the ``$-$'' sign.  
\end{remark}

\begin{proof}
Part (2) is straightforward from the definitions.
Let us first prove the remaining claims in 
the case when $G$ is a plabic graph.
In fact, in this case this construction gives exactly
the {\it boundary measurement parametrization\/} 
of $\Pi_G$ from \cite[Section 5]{Pos2}.
The Pl\"ucker coordinates for the boundary measurement parametrization
were given in \cite[Proposition~5.3]{Pos2} and expressed
by Talaska \cite[Theorem~1.1]{Talaska} in terms of network flows
on the graph $G$.  
The construction of the boundary measurement parametrization
(and Talaska's formula) depends on a choice 
of a reference perfect orientation $\O_0$.  
One observes that any other perfect orientation $\O$ of the plabic 
graph $G$ is obtained from $\O_0$ by 
reversing the edges along a network flow,
which gives a bijection between network flows and perfect orientations of $G$.
This shows that the above expression for $\Delta_I$ 
is equivalent to Talaska's formula, which proves the equivalence 
of the above perfect orientation parametrization and the boundary measurement
parametrization from \cite{Pos2}.  
Parts (1) and (3) now follow from results of \cite{Pos2}.

For an arbitrary Grassmannian graph $G'$, let $G$ be a plabic graph
that refines $G'$.  We already know that each 
``little plabic graph'' 
$G_v$, i.e., the subgraph of $G$ that refines a vertex $v$ of $G'$, 
parametrizes each positive vertex-Grassmannian $Gr^{>0}(v)$
by a birational subtraction-free bijection $\bar\mu_v:=\bar\mu_{G_v}$.
We also know the map $\bar\mu_G$ for the plabic graph $G$ parametrizes 
the cell $\Pi_{\M(G)}$ if $G$ reduced, or maps surjectively but not bijectively
onto $\Pi_{\M(G)}$ if $G$ is not reduced.
Then the map $\bar\mu_{G'}$ is given by the composition $\bar\mu_{G}\circ 
(\bigtimes_{v\in V_\inte} (\bar \mu_v)^{-1})$ and the needed result follows.
\end{proof}

The above construction can be thought of as 
a gluing of the ``big Grassmannian'' out of ``little
Grassmannians.'' This is similar to a tiling of a big geometric object (polytope)
by smaller pieces (smaller polytopes).  As we will see, this construction
literally corresponds to certain subdivisions of polytopes.

\section{Polyhedral subdivisions: Baues poset and fiber polytopes}

In this section we discuss Billera-Sturmfels' theory \cite{BS} of {\it fiber
polytopes,} the {\it generalized Baues problem\/} \cite{BKS},
and {\it flip-connectivity,} see also \cite{Rei, RS, ARS, Ath, AS} 
for more details.

\subsection{The Baues poset of $\pi$-induced subdivisions}

Let $\pi:P\to Q$ be an affine projection from one convex polytope $P$
to another convex polytope $Q=\pi(P)$.

Informally, a $\pi$-induced polyhedral subdivision is a collection of faces of
the polytope $P$ that projects to a polyhedral subdivision of the polytope $Q$.

Here is a rigorous definition, see \cite{BS}.  Let $A$ be the multiset of projections
$\pi(v)$ of vertices $v$ of $P$. Each element $\pi(v)$ of $A$ is labelled by
the vertex $v$. 
For $\sigma\subset A$, let $\conv(\sigma)$ denotes the convex hull of $\sigma$.
We say that $\sigma'\subset \sigma$ is a {\it face\/} of $\sigma$ if $\sigma'$
consists of all elements of $\sigma$ that belong to a face of the polytope
$\conv(\sigma)$.

A {\it $\pi$-induced subdivision\/} is a finite collection $S$ of 
subsets $\sigma\subset A$, called {\it cells}, such that
\begin{enumerate}
\item Each $\sigma\in S$ is the projection under $\pi$ of the 
vertex set of a face of $P$.
\item For each $\sigma\in S$, $\dim(\conv(\sigma)) = \dim(Q)$.
\item For any $\sigma_1,\sigma_2\in S$, 
$\conv(\sigma_1)\cap \conv(\sigma_2) = \conv(\sigma_1\cap \sigma_2)$.
\item  For any $\sigma_1,\sigma_2\in S$, $\sigma_1\cap \sigma_2$ 
is either empty or a face of both $\sigma_1$ and $\sigma_2$.
\item $\bigcup_{\sigma\in S} \conv(\sigma) = Q$.
\end{enumerate}

The {\it Baues poset\/} $\omega(P\overset\pi\to Q)$ is the poset
of all $\pi$-induced subdivisions partially ordered by {\it refinement,}
namely, $S\leq T$ means that, for every cell $\sigma\in S$,
there exists a cell $\tau \in T$ such that $\sigma \subset \tau$.
This poset has a unique maximal element $\hat 1$,
called the {\it trivial subdivision}, that consists of 
a single cell $\sigma=A$.  All other elements are called {\it proper\/} 
subdivisions.  Let $\hat\omega(P\overset\pi\to Q)):= 
\omega(P\overset\pi\to Q)-\hat 1$ be 
the poset of proper $\pi$-induced subdivisions
obtained by removing the maximal element $\hat 1$.
The minimal elements of the Baues poset are called {\it tight\/} 
$\pi$-induced subdivisions.

Among all $\pi$-induced subdivisons, there is a subset of {\it coherent\/}
subdivisions that come from linear {\it height functions\/} $h:P\to \R$ as
follows.  For each $q\in Q$, let $\bar F_q$ be the face of the fiber
$\pi^{-1}(q)\cap P$ where the height function $h$ reaches its maximal value.  The face
$\bar F_q$ lies in the relative interior of some face $F_q$ of $P$.  The
collection of faces $\{F_q\}_{q\in Q}$ projects to a 
$\pi$-induced subdivision of $Q$.  
Let $\omega_\coh(P\overset\pi\to Q)\subseteq\omega(P\overset\pi\to Q)$ be the
subposet of the Baues poset formed by the coherent $\pi$-induced subdivisions.
This coherent part of the Baues poset is isomorphic to the face
lattice of the convex polytope 
$\Sigma(P\overset\pi\to Q)$, called the {\it fiber polytope,}
defined as the Minkowskii integral of fibers of $\pi$
(the limit Minkowskii sums):
$$
\Sigma(P\overset\pi\to Q)
:=\int_{q\in Q} (\pi^{-1}(q)\cap P) \, dq
$$
In general, the whole Baues poset $\omega(P\overset\pi\to Q)$
may not be polytopal.

\subsection{The generalized Baues problem and flip-connectivity}

For a finite poset $\omega$, the {\it order complex\/} $\Delta\omega$
is the simplicial complex of all chains in $\omega$.
The ``topology of a poset $\omega$'' means the topology of the simplicial
complex $\Delta\omega$.
For example, if $\omega$ is the face poset of a regular cell complex 
$\Delta$, then $\Delta\omega$ is the barycentric subdivision 
of the cell complex $\Delta$; 
and, in particular, $\Delta\omega$ is homeomorphic to $\Delta$.

Clearly, the subposet $\hat \omega_\mathrm{coh}(P\overset\pi\to Q)$ 
of proper coherent $\pi$-induced subdivisions homotopy equivalent to a
$(\dim(P) - \dim(Q) -1)$-sphere, because it is the face lattice
of a convex polytope
of dimension $\dim(P)-\dim(Q)$, 
namely, the fiber polytope $\Sigma(P\overset\pi\to Q)$.

The {\it generalized Baues problem\/} (GBP) posed by Billera, Kapranov,
Sturmfels \cite{BKS} asks whether the same is true about the poset 
of all proper $\pi$-induced subdivisions.
Is it true that $\hat \omega(P\overset\pi\to Q)$ is 
homotopy equivalent to a $(\dim(P) - \dim(Q) -1)$-sphere?
In general, the GBP is a hard question.  Examples of Baues posets with 
disconnected topology were constructed by Rambau and Ziegler \cite{RZ} and more 
recently by Liu \cite{Liu}.  There are, however, several general classes
of projections of polytopes, where the GBP has an affirmative answer,
see the next section.  

Another related question is about connectivity by flips.
For a projection $\pi:P\to Q$, the {\it flip graph\/}
is the restriction of the Hasse
diagram of the Baues poset $\omega(P\overset\pi\to Q)$ to elements
of rank $0$ (tight subdivisions) and rank $1$ (subdivisions that cover
a tight subdivision).  The elements of rank 1 in the flip graph 
are called {\it flips.}  The flip-connectivity problem asks 
whether the flip graph is connected.
The coherent part of the flip graph is obviously connected, because
it is the $1$-skeleton of the fiber polytope $\Sigma(P\overset\pi\to Q)$.

The GBP and the flip-connectivity problem are related to each other, 
but, strictly speaking, neither of them implies the other,
see \cite[Section~3]{Rei} for more details.

\subsection{Triangulations and zonotopal tilings} 

There are two cases of the above general setting that 
attracted a special attention in the literature.

The first case is when the polytope $P$ is the $(n-1)$-dimensional {\it simplex\/} 
$\Delta^{n-1}$.  The multiset $A$ of projections of vertices of the simplex
can be an arbitrary multiset of $n$ points, and $Q=\conv(A)$ can be an 
arbitrary convex polytope.
In this case, the Baues poset $\omega(\Delta^{n-1}\overset\pi\to Q)$ is the 
poset of all polyhedral subdivisions of $Q$ (with vertices at $A$);
tight $\pi$-induced subdivisions are {\it triangulations\/} of $Q$;
and the fiber polytope $\Sigma(\Delta^{n-1}\overset\pi\to Q)$ is exactly
is the {\it secondary polytope\/} of Gelfand-Kapranov-Zelevinsky \cite{GKZ},
which appeared in the study of discriminants.

In particular, for a projection 
of the simplex $\Delta^{n-1}$ to an $n$-gon $Q$, $\pi$-induced subdivisions are exactly 
the subdivisions of the $n$-gon by noncrossing chords.  All of them are coherent.
Tight subdivisions are triangulations of the $n$-gon.
There are the Catalan number
$C_{n-2} = {1\over n-1 } {2n-4\choose n-2}$ of triangulations of the $n$-gon.
The fiber polytope (or the secondary polytope)
in this case is the Stasheff {\it associahedron.}

Another special case is related to projections $\pi:P\to Q$
of the {\it hypercube\/} $P=\cube_n:=[0,1]^n$.  The projections $Q=\pi(\cube_n)$ 
of the hypercube form a special class of polytopes, called {\it zonotopes.}
In this case, $\pi$-induced subdivisions are {\it zonotopal tilings\/} of 
zonotopes $Q$.  According to Bohne-Dress theorem \cite{Boh}, 
zonotopal tilings of $Q$ are in bijection with {\it 1-element extensions\/} of 
the oriented matroid associated with the zonotope $Q$.

For a projection of the $n$-hypercube $\cube_n$ to a $1$-dimensional line segment,
the fiber polytope is the {\it permutohedron.}
For a projection $\pi:\cube_n\to Q$ of the $n$-hypercube $\cube_n$ to a $2n$-gon $Q$,
{\it fine zonotopal tilings} (i.e., tight $\pi$-induced subdivisions)
are known as {\it rhomus tilings\/} of the $2n$-gon.  They correspond to
commutation classes of {\it reduced decompositions\/} of the longest permutation
$w_\circ\in S_n$.

\section{Cyclic polytopes and cyclic zonotopes}

Fix two integers $n$ and $0\leq d\leq n-1$.

\begin{definition}
A {\it cyclic projection\/} is a linear map 
$$
\pi:\R^n\to \R^{d+1},\quad  \pi:x\mapsto Mx
$$ 
given by a $(d+1)\times n$ matrix 
$M=(u_1,\dots,u_n)$ (the $u_i$ are the column vectors)
with all positive maximal $(d+1)\times (d+1)$ minors and such that 
$f(u_1)=\cdots=f(u_n)=1$ for some linear form $f:\R^{d+1}\to\R$.
In other words, $M$ represents a point of
the positive Grassmannian $Gr^{>0}(d+1,n)$ with columns $u_i$ rescaled
so that they all lie on the same affine hyperplane $H_1=\{y\in\R^{d+1}\mid f(y)=1\}$.
\end{definition}

The {\it cyclic polytope\/} is 
the image under a cyclic projection $\pi$ of the standard $(n-1)$-dimensional
{\it simplex\/}
$\Delta^{n-1}:=\conv(e_1,\dots,e_n)$
$$
C(n,d) := \pi(\Delta^{n-1}) \subset H_1.
$$
The {\it cyclic zonotope\/} is the image of the 
standard {\it $n$-hypercube\/} $\cube_n:=[0,1]^n\subset\R^n$ 
$$
Z(n,d+1):=\pi(\cube_n)\subset \R^{d+1}.
$$
Remark that, for each $n$ and $d$, there are many 
combinatorially (but not linearly) isomorphic cyclic polytopes $C(n,d)$
and cyclic zonotopes $Z(n,d+1)$ that depend on a choice of the cyclic projection $\pi$.
Clearly, $C(n,d)=Z(n,d+1)\cap H_1$.

Ziegler \cite{Ziegler}  identified {\it fine zonotopal tilings\/}
of the cyclic zonotope $Z(n,d+1)$, i.e., 
the minimal elements of the Baues poset 
$\omega(\cube_n\overset\pi\to Z(n,d+1))$, 
with elements of Manin-Shekhtman's {\it higher
Bruhat order\/} \cite{MS}, also studied by Voevodsky and Kapranov \cite{VK}.  
According to results of Sturmfels and Ziegler \cite{SZ}, Ziegler \cite{Ziegler},
Rambau \cite{Ramb}, and Rambau and Santos \cite{RS}, the GBP and flip-connectivity have 
affirmative answers in these cases.

\begin{theorem} 
{\rm (1)} \cite{SZ}
For  $\pi:\cube_n \to Z(n,d+1)$, the poset of proper zonotopal 
tilings of the cyclic zonotope $Z(n,d+1)$ is homotopy equivalent
to an $(n-d-2)$-dimensional sphere.
The set of fine zonotopal tilings of $Z(n,d+1)$ is connected by flips.

\smallskip

{\rm (2)} \cite{RS} For $\pi:\Delta^{n-1}\to C(n,d)$,
the poset of proper subdivisions of the cyclic polytope $C(n,d)$
is homotopy equivalent to an $(n-d-2)$-dimensional sphere.
\cite{Ramb} The set of triangulations of the cyclic polytope 
$C(n,d)$ is connected by flips.

\end{theorem}

\section{Cyclic projections of the hypersimplex}

Fix three integers $0\leq k\leq n$ and $0\leq d\leq n-1$.
The {\it hypersimplex\/} 
$\Delta_{kn}:=\conv\left\{e_I\mid I\in{[n]\choose k}\right\}$
is the {\it $k$-th section\/} of the $n$-hypercube $\cube_n\subset\R^n$
$$
\Delta_{kn}
=\cube_n\cap \{x_1+\cdots+x_n=k\}.
$$
Let $\pi:\R^n\to\R^{d+1}$ be a cyclic projection as above.
Define the polytope
$$
Q(k,n,d):=\pi(\Delta_{kn}) = Z(n,d+1)\cap H_k,
$$
where $H_k$ is the affine hyperplane $H_k:=\{y\in\R^{d+1}\mid f(y)=k\}$.
Clearly, for $k=1$, the polytope $Q(1,n,d)$ is the cyclic polytope
$C(n,d)$.

Let $\omega(k,n,d)$ be the Baues poset of $\pi$-induced subdivisions for 
a cyclic projection $\pi:\Delta_{kn}\to Q(k,n,d)$:
$$
\omega(k,n,d):=\omega(\Delta_{kn}\overset\pi\to Q(k,n,d)).
$$
Let $\omega_\coh^\pi(k,n,d):=\omega_\coh(\Delta_{kn}\overset\pi\to Q(k,n,d))
\subseteq \omega(k,n,d)$
be its coherent part.
Note that the coherent part $\omega_\coh^\pi(k,n,d)$ depends on a choice
of the cyclic projection $\pi$, but the whole poset $\omega(k,n,d)$ is independent
of any choices.
The coherent part 
$\omega_\coh^\pi(k,n,d)$ may not be equal $\omega(k,n,d)$.
For example they are not equal for $(k,n,d)=(3,6,2)$.

The poset $\omega(k,n,d)$ is a generalization of the Baues poset 
of subdivisions of the cyclic polytope $C(n,d)$,
and is related to the Baues poset of zonotopal 
tilings of the cyclic zonotope $Z(n,d+1)$ in an obvious manner.
For $k=1$, $\omega(1,n,d)=\omega(\Delta^{n-1}\overset\pi\to C(n,d))$.  
For any $k$, there is the order preserving 
{\it $k$-th section map}
$$
\mathrm{Section}_k:  \omega(\cube_n\overset\pi\to Z(n,d+1))\to \omega(k,n,d)
$$
that send a zonotopal tiling of $Z(n,d+1)$ to its section by the hyperplane $H_k$.

Let $\omega_{\lift}(k,n,d) \subseteq \omega(k,n,d)$ be the image of the
map $\mathrm{Section}_k$. 
We call the elements of $\omega_\lift(k,n,d)$ the {\it lifting\/} 
$\pi$-induced subdivisions.
They form  the subset of $\pi$-induced subdivisions from $\omega(k,n,d)$ 
that can be lifted to a zonotopal
tiling of the cyclic zonotope $Z(n,d+1)$.  
Clearly, we have
$$
\omega_\coh^\pi(k,n,d) \subseteq \omega_\lift(k,n,d)\subseteq \omega(k,n,d).
$$

The equality of the sets of minimal elements of $\omega(k,n,d)$ and
$\omega_\lift(k,n,d)$ was proved
in the case $k=1$ by Rambau and Santos \cite{RS}, 
who showed that all triangulations of the cyclic polytope 
$C(n,d)$ are lifting triangulations. 
For $d=2$, the equality follows from the result of Galashin \cite{Gal} 
(Theorem~\ref{th:galashin_cubes} below) 
about plabic graphs, as we will explain in the next section.

\begin{theorem}  
The minimal elements of the posets $\omega_{\lift}(k,n,d)$ 
and $\omega(k,n,d)$ are the same in the following cases:
{\rm (1)} 
$k=1$ and any $n, d$;
{\rm (2)}
$d=2$ and any $k, n$.
\end{theorem}

Flip-connectivity \cite{SZ, Ziegler} of zonotopal tilings
of $Z(n,d+1)$ easily implies the following claim.

\begin{lemma}
The minimal elements of $\omega_\lift(k,n,d)$ are connected by flips.
\end{lemma}

Indeed, for any pair of fine zonotopal tilings $T$ and $T'$ of $Z(n,d+1)$ 
connected by a flip, their $k$-sections $\mathrm{Section}_k(T)$ and 
$\mathrm{Section}_k(T')$ are either equal to each other or connected by a flip.

The Baues posets of the form $\omega(k,n,d)$ are good candidates 
for a general class of projections of polytopes 
where the GBP and flip-connectivity problem might have affirmative answers.

\begin{problem}
Is the poset $\omega(k,n,d)-\hat 1$ homotopy equivalent of a sphere?  
Can its minimal elements be connected by flips?
Is it true that $\omega_{\lift}(k,n,d)=\omega(k,n,d)$?
\end{problem}

\begin{example}
For $d=1$, the Baues poset $\omega(k,n,1)$ is already interesting.
Its minimal elements correspond to {\it monotone paths\/} on the hypersimplex
$\Delta_{kn}$, which are increasing paths that go along the edges of the 
hypersimplex $\Delta_{kn}$.  
Such paths are the subject of the original (non-generalized) Baues problem \cite{B},
which was proved by Billera, Kapranov, Sturmfels \cite{BKS}
(for any $1$-dimensional projection of a polytope).
More specifically, monotone paths on $\Delta_{kn}$ correspond to directed paths 
from $[1,k]$ to $[n-k+1,n]$ in the directed graph
on ${[n]\choose k}$ with edges $I\to J$ if $J=(I\setminus\{i\})\cup \{j\}$ for $i<j$.

It is not hard to see that, for $k=1,n-1$, there are $2^{n-2}$ monotone paths,
and the posets $\omega(1,n,1)$ and $\omega(n-1,n,1)$ are isomporphic to 
the Boolean lattice $B_{n-2}$, i.e., the face poset of the hypercube $\cube_{n-2}$.  
For $n=2,3,4,5$, the Baues poset $\omega(2,n,1)$ has $1,2,10,62$ minimal elements.

Monotone paths on $\Delta_{kn}$ might have different lengths.  
The {\it longest\/} monotone paths are in an easy bijection with {\it standard
Young tableaux\/} of the rectangular shape $k\times (n-k)$.  
By the hook-length formula, their number is 
$(k(n-k))! \prod_{i=0}^{n-k-1}{i!\over (k+i)!}$.

Note, however, $\omega(k,n,d)\ne \omega_\lift(k,n,d)$ for $(k,n,d)=(2,5,1)$.
Indeed, Galashin pointed out that the monotone path 
$\{1,2\}\to\{1,3\}\to\{1,4\}\to\{2,4\}\to\{3,4\}$ {\it cannot\/} be lifted
to a rhombus tiling of the the $2n$-gon $Z(n,2)$, because it is not 
weakly separated.
\end{example}

\section{Grassmannian graphs as duals of polyhedral subdivisions induced by 
projections of hypersimplices}

Let us now discuss the connection between the positive
Grassmannian and combinatorics of polyhedral subdivisions.  
In fact, the positive Grassmannian
is directly related to the setup of the previous section for $d=2$.

\begin{theorem}
The poset of complete reduced Grassmannian graphs of type $(k,n)$ ordered 
by refinement is canonically isomorphic to the Baues poset $\omega(k,n,2)$
of $\pi$-induced subdivisions for a 2-dimensional cyclic projection $\pi$
of the hypersimplex $\Delta_{kn}$.
Under this isomorphism, plabic graphs correspond to tight $\pi$-induced
subdivisions and moves of plabic graphs correspond to flips between 
tight $\pi$-induced subdivisions.
\end{theorem}

Theorem~\ref{th:move_refinement_equivalence}(2) 
(\cite[Theorem~13.4]{Pos2}) immediately implies flip-connectivity.

\begin{corollary}
The minimal elements of Baues poset $\omega(k,n,2)$ are
connected by flips.
\end{corollary}

\begin{example}
The Baues poset $\omega(1,n,2)$ is the poset of subdivisions
of an $n$-gon by non-crossing chords, i.e., it is the Stasheff's 
{\it associahedron.}  Its minimal elements correspond to the 
Catalan number ${1\over n-1}{2n-4\choose n-2}$ triangulations of 
the $n$-gon.
\end{example}

We can think of the Baues posets $\omega(k,n,2)$ as some 
kind of ``generalized associahedra.''  In general, they are not polytopal.
But they share some nice features with the associahedron.
It is well-known that every face of the associahedron is a direct product
of smaller associahedra.  The same is true for all $\omega(k,n,2)$.

\begin{proposition}
For any element $S$ in  $\omega(k,n,2)$, the lower order interval
$\{S' \mid S'\leq S\}$ in the Baues poset 
$\omega(k,n,2)$
is a direct product of Baues posets of the same form $\omega(k',n',2)$.
\end{proposition}

\begin{proof}
This is easy to see in terms of complete reduced Grassmannian graphs $G$.
Indeed, for any $G$, all refinements of $G'$ are obtained
by refining all vertices of $G$ independently from each other.
\end{proof}

This property is related to the fact that every face of the 
hypersimplex $\Delta_{kn}$ is a smaller hypersimplex, as we discuss below.

\begin{remark}
Among all reduced Grassmannian/plabic graphs, there is a subset
of {\it coherent\/} (or {\it regular}) graphs, namely the ones 
that correspond to the coherent $\pi$-induced subdivisions
from $\omega_\coh(k,n,2)$.
Each of these graphs can be explicitly constructed in terms of 
a {\it height function.}   This subclass depends on a choice of the 
cyclic projection $\pi$.  Regular plabic graphs
are related to the study of {\it soliton solutions\/}
of Kadomtsev-Petviashvili (KP) equation, see \cite{KoW11, KoW12}.
We will investigate the class of regular plabic graphs in \cite{GPW}.
\end{remark}

Let us now give more details on the correspondence between Grassmannian 
graphs and subdivisions.
A cyclic projection $\pi:\Delta_{kn}\to Q(k,n,2)$
is the linear map
given by a $3\times n$ matrix $M=(u_1,\dots,u_n)$
such that $[u_1,\dots,u_n]\in Gr^{>0}(3,n)$ 
and $u_1,\dots,u_n$ all lie on the same affine plane $H_1\subset\R^3$.
Without loss of generality, assume that $H_1=\{(x,y,z)\mid z=1\}$.
The positivity condition means that the points $\pi(u_1),\dots,\pi(u_n)$
form a convex $n$-gon with vertices arranged in the counterclockwise order.

The polytope $Q:=Q(k,n,2)=\pi(\Delta_{kn})$ is the 
convex $n$-gon in the affine plane $H_k=\{(x,y,k)\}\subset \R^3$ with 
the vertices $\pi(e_{[1,k]})$, $\pi(e_{[2,k+1]})$, \dots, $\pi(e_{[n,k-1]})$
(in the counterclockwise order)
corresponding to all consecutive cyclic intervals of size $k$ in $[n]$.

Notice that each face $\gamma$ of the hypersimplex $\Delta_{kn}$ 
is itself a smaller hypersimplex of the form 
$$
\gamma_{I_0,I_1}:=\{(x_1,\dots,x_n)\in\Delta_{kn} \mid x_i = 0 
\textrm{ for } i\in I_0, \ x_j=1\textrm{ for }j\in I_1\}
$$
where $I_0$ and $I_1$ are disjoint subsets of $[n]$.
So $\gamma\simeq \Delta_{hm}$, where $h=k-|I_0|$ and $m=n-|I_0|-|I_1|$.
The projection $\pi$ maps the face $\gamma$ to the 
$m$-gon $\pi(\gamma)$ that carries an additional parameter $h$.

Thus the $\pi$-induced subdivisions $S$ are in bijective correspondence with 
the tilings of the $n$-gon $Q$ by smaller convex polygons such that: 
\begin{enumerate}
\item Each vertex has the form $\pi(e_I)$ for $I\in{[n]\choose k}$.
\item  Each edge has the form $[\pi(e_I),\pi(e_J)]$ for
two $k$-element subsets $I$ and $J$ such that $|I\cap J|=k-1$.
\item  Each face is an $m$-gon of the form $\pi(\gamma_{I_0,I_1})$, as above.
\end{enumerate}

Let $S^*$ be the planar dual of such a tiling $S$.  
The graph $S^*$ has exactly $n$
boundary vertices $b_i$ corresponding to the sides 
$[\pi(e_{[i,i+k-1]}), \pi(e_{[i+1, i+k])}]$ of the $n$-gon $Q$.
The internal vertices $v$ of $S^*$ 
(corresponding to faces $\gamma$ of $S$) are equipped with 
the parameter $h=h(v)\in\{0,\dots,\deg(v)\}$.  
Thus $S^*$ has the structure of a Grassmannian graph.
Moreover,  each face $F$ of $S^*$ (corresponding a vertex $\pi(e_I)$ of $S$)
is labelled by a subset $I\in{[n]\choose k}$.  
We can now make the previous theorem more precise.

\begin{theorem}
The map $S\mapsto S^*$ is 
an isomporphism between the Baues poset $\omega(k,n,2)$
and complete reduced Grassmannian graphs $G$ of type $(k,n)$.
For each face $F$ of $G=S^*$ corresponding to a vertex $\pi(e_I)$ of $S$,
the subset $I\subset {[n]\choose k}$ is exactly the face label $I_F$ 
(see Definition~\ref{def:face_label}).
\end{theorem}

\begin{proof}
Let us first show that {\it tight\/} $\pi$-induced subdivisions $S$ are 
in bijection with complete reduced {\it plabic\/} graphs of type $(k,n)$.
That means that, in addition to the conditions (1), (2), (3) above,
we require the tiling $S$ of the $n$-gon $G$ has all triangular faces.
So $S$ is a triangulation of the $n$-gon $Q$ of a special kind,
which we call a {\it plabic triangulation.}

Such plabic triangulations of the $n$-gon are closely related to 
{\it plabic tilings\/} from \cite{OPS}.
The only difference between plabic triangulations and plabic tilings 
is that the latter correspond not to (3-valent) plabic graphs 
(as defined in the current paper) but to {\it bipartite\/} plabic graphs.
A bipartite plabic graph $G$ is exactly a Grassmannian graph such 
that each internal vertex either has type $(1,d)$ (white vertex)
or type $(d-1,d)$ (black vertex), and every edge of $G$ connects vertices
of different colors.  Each reduced 3-valent plabic graph $G'$ can be easily 
converted into a bipartite plabic graph $G$ by contrating edges
connecting vertices of the same color.
It was shown in \cite[Theorem~9.12]{OPS} that the planar dual graph
of any reduced bipartite plabic graph $G$ can be embedded inside an $n$-gon
as a plabic tiling with black and white regions 
and all vertices of the form $\pi(e_I)$.
If we now subdivide the black and white regions of such 
plabic tiling by chords into triangles, we can get back the plabic triangulation
associated with a (3-valent) plabic graph $G'$.
This shows that any complete reduced (3-valent) plabic graph is indeed the planar dual
of a tight $\pi$-induced subdivision.

On the other hand, for each plabic triangulation $S$ we can construct
the plabic graph by taking its planar dual $G=S^*$ as described above.
It is easy to check from the definitions that the decorated
strand permutation $w$ of $G$ is exactly $w(i)= i+k\pmod n$.
It remains to show that this plabic graph $G$ is reduced.
Suppose that $G$ is not reduced.  
Then by Theorem~\ref{th:move_refinement_equivalence}(5),
after possibly applying a sequence of moves $(1,4)$, $(2,4)$, and/or $(3,4)$,
we get a plabic graph with a pair of parallel edges or with a loop-edge.
It is straightforward to check that applying the moves 
$(1,4)$, $(2,4)$, $(3,4)$ corresponds to local transformations 
of the plabic triangulation $S$, and transforms it into another
plabic triangulations $S'$.  However, it is clear
that if a plabic graph $G$ contains parallel edges of a loop-edge, 
then the dual graph is {\it not\/} a plabic triangulation.
So we get a contradiction, which proves the result for plabic graphs 
and tight subdivisions.

Now let $G'$ be any complete reduced Grassmannian graph of type $(k,n)$, 
and let $G$ be its plabic refinement.  We showed that we can embed the planar 
dual graph $G^*$ as a plabic triangulation $S$ into the $n$-gon.
The union of triangles in $S$ that correspond to a single vertex $v$
of $G'$ covers a region inside $Q$.  We already know that this region 
is a convex $m$-gon 
(because we already proved the correspondence for plabic graphs).
Thus, for each vertex of $G'$, we get a convex polygon in $Q$ and all these 
polygons form $\pi$-induced subdivision.  So we proved that the planar
dual of $G'$ can be embedded as a polyhedral subdivision of $Q$.  The inverse
map is $S'\mapsto G'=(S')^*$.  
\end{proof}

Let us mention a related result of Galashin~\cite{Gal}.

\begin{theorem} 
\label{th:galashin_cubes}
\cite{Gal}
Complete reduced plabic graphs of type $(k,n)$ are exactly 
the dual graphs of sections of fine zonotopal tilings
of the 3-dimensional cyclic zonotope $Z(n,3)$ by the hyperplane $H_k$.
\end{theorem}

In view of the discussion above, this result means that
any tight $\pi$-induced subdivision in $\omega(k,n,2)$ can be lifted to a
fine zonotopal tiling of the cyclic zonotope $Z(n,3)$.
In other words, the posets $\omega(k,n,2)$ and $\omega_\lift(k,n,2)$
have the same sets of minimal elements.
A natural question to ask: Is the same true for all (not necessarily 
minimal) elements of $\omega(k,n,2)$?

\section{Membranes and discrete Plateau's problem}

Membranes from the joint project with Lam \cite{LP} provide another related 
interpretation of plabic graphs.
Let $\Phi=\{e_i - e_j \mid i\ne j\}\in\R^n$, where $e_1,\dots,e_n$
are the standard coordinate vectors. 

\begin{definition} \cite{LP}
A {\it loop\/} $L$ is a closed piecewise-linear curve in $\R^n$ 
formed by line segments $[a,b]$ such that $a,b\in\Z^n$ and $a-b\in\Phi$.

A {\it membrane\/} $M$ with {\it boundary loop\/} $L$ is an embedding 
of a 2-dimensional disk into $\R^n$ such that $L$ is the boundary of $M$,
and $M$ is made out of triangles $\conv(a,b,c)$, where $a,b,c\in\Z$
and $a-b,b-c,a-c\in\Phi$.

A {\it minimal membrane\/} $M$ is a membrane that has minimal possible area
(the number of triangles) among all membranes with the same boundary loop $L$.
\end{definition}

The problem about finding a minimal membrane $M$
with a given boundary loop $L$ is a discrete version 
of {\it Plateau's problem\/} about minimal surface.
Informally speaking, membranes correspond to (the duals of) plabic graphs,
and minimal membranes correspond to reduced plabic graphs.   
Here is a more careful claim.

\begin{theorem} \cite{LP}
Let $w\in S_n$ be a permutation without fixed points 
with helicity $h(w)=k$.
Let $L_w$ be the closed loop inside the hypersimplex $\Delta_{kn}$
formed by the line segments $[a_1,a_2],[a_2,a_3],\dots,[a_{n-1},a_n],[a_n,a_1]$
such that $a_{i+1}-a_i = e_{w(i)}-e_i$, for $i=1,\dots,n$, with indices 
taken modulo $n$.

Then minimal membranes $M$ with boundary loop $L_w$ are in bijection
with reduced plabic graphs $G$ with strand permutation $w$.
Explicitly, the correspondence is given as follows.  Faces $F$ of $G$ with face
labels $I=I_F$ correspond to vertices $e_I$ of the membrane $M$.  
Vertices of $G$ with 3 adjacent faces
labeled by $I_1,I_2,I_3$ correspond to triangles $\conv(e_{I_1},e_{I_2},e_{I_3})$
in $M$.

Moves of plabic graphs correspond to local area-preserving transformations of
membranes.  
Any two minimal membranes with the same boundary loop $L_w$
can be obtained from each other by these local transformations.
\end{theorem}

\section{Higher positive Grassmannians and Amplituhedra}

The relation between the positive Grassmannian 
$Gr^{>0}(k,n)$ and the Baues poset $\omega(k,n,2)$
raises a natural question: What is the geometric counterpart of 
the Baues poset $\omega(k,n,d)$ for any $d$?
These ``higher positive Grassmannians'' should generalize $Gr^{>}(k,n)$
in the same sense as Manin-Shekhtman's higher Bruhat orders
generalize the weak Bruhat order.
The first guess is that they might be related to amplituhedra.

Arkani-Hamed and Trnka \cite{AT}, motivated by the study of scattering amplitudes
in $\mathcal{N}=4$ supersymmetric Yang-Mills (SYM) theory, defined
the {\it amplituhedron\/} $A_{n,k,m} = A_{n,k,m}(Z)$ 
as the image of the nonnegative Grassmannian
$Gr^{\geq}(k,n)$ under the ``linear projection''
$$
\tilde Z: Gr^{\geq 0}(k,n)\to Gr(k,k+m),\quad
[A]\mapsto [A\,Z^T]
$$
induced by a totally positive $(k+m)\times n$ matrix $Z$, for $0\leq m\leq n-k$.
The case $m=4$ is of importance for physics.

In general, the amplituhedron $A_{n,k,m}$ has quite mysterious geometric
and combinatorial structure.
Here are a few special cases where its structure was understood better.
For $m=n-k$, $A_{n,k,n-k}$ is isomorphic to the nonnegative Grassmannian 
$Gr^{\geq 0}(k,n)$.  For $k=1$, $A_{n,1,m}$ is
(the projectivization of) the cyclic polytope $C(n,m)$.
For $m=1$, Karp and Williams \cite{KaW} showed that the structure
of the amplituhedron $A_{n,k,1}$ is equivalent to the complex of bounded
regions of the cyclic hyperplane arrangement.
In general, the relationship between the amplituhedron $A_{n,k,m}$ 
and polyhedral subdivisions is yet to be clarified.

\end{document}